\documentclass[11pt,twoside,a4paper]{article}
     
\usepackage{amssymb}
\usepackage{amsmath}
\usepackage{amsthm, color}

\allowdisplaybreaks

\usepackage{amsfonts}
\usepackage{bbm}
\usepackage{verbatim}
\usepackage{url}
\usepackage[T1]{fontenc}
\usepackage{mathrsfs}

\usepackage{hyperref}
\hypersetup{
    colorlinks=true,
    linkcolor=blue,
    citecolor=blue,
    filecolor=magenta,      
    urlcolor=cyan,
}

\DeclareUnicodeCharacter{0301}{\"u}

\newcommand{\R}{\bbR}

\newcommand{\A}{\mathcal A}

\newcommand{\D}{\mathcal D}

\newcommand{\RR}{\mathcal R}

\newcommand{\Be}{\begin{equation}}
\newcommand{\Ee}{\end{equation}}

\newcommand{\Bm}{\begin{multline}}
\newcommand{\Em}{\end{multline}}
\newcommand{\Bea}{\begin{eqnarray}}
\newcommand{\Eea}{\end{eqnarray}}
\newcommand{\Beas}{\begin{eqnarray*}}
\newcommand{\Eeas}{\end{eqnarray*}}
\newcommand{\Benu}{\begin{enumerate}}
\newcommand{\Eenu}{\end{enumerate}}
\newcommand{\Bi}{\begin{itemize}}
\newcommand{\Ei}{\end{itemize}}

\def\intslash{\rlap{\kern  .32em $\mspace {.5mu}\backslash$ }\int}
\def\qsl{{\rlap{\kern  .32em $\mspace {.5mu}\backslash$ }\int_{Q_x}}}

\def\emph#1{{\it #1 }}

\def\Rn{{\mathbb R^n}}
\def\Rd{{\mathbb R^d}}

\def\supp{{\text{\rm supp}}}

\def\R{{\hbox{\bf R}}}

 at 10 true pt

\newcommand{\G}{\Rd}

\def\be#1{\begin{equation}\label{ #1}}
\def\endeq{\end{equation}}
\def\endal{\end{align}}
\def\bas{\begin{align*}}
\def\eas{\end{align*}}
\def\bi{\begin{itemize}}
\def\ei{\end{itemize}}

\def\Rd{{\hbox{\bf R}_d}}

\def\emph#1{{\it #1}}
\def\textbf#1{{\bf #1}}

\begin{document}

\newtheorem{theorem}{Theorem}[section]
\newtheorem{prop}[theorem]{Proposition}
\newtheorem{lemma}[theorem]{Lemma}
\newtheorem{definition}[theorem]{Definition}
\newtheorem{corollary}[theorem]{Corollary}
\newtheorem{example}[theorem]{Example}
\newtheorem{remark}[theorem]{Remark}
\newcommand{\ra}{\rightarrow}
\renewcommand{\theequation}
{\thesection.\arabic{equation}}
\newcommand{\ccc}{{\cal C}}

\allowdisplaybreaks

\def\HAL { H^1_{{at}, L_1, L_2, M}({\Bbb R}^{n_1}\times{\Bbb R}^{n_2})}
 \def\HSL { H^1_{L_1, L_2, S}({\Bbb R}^{n_1}\times{\Bbb R}^{n_2}) }
\def\HL  {   H^1_{L_1, L_2}({\Bbb R}^{n_1}\times{\Bbb R}^{n_2}) }
\def\HLD  {  H^{\rm 1}_{L} \cap L^2   }
\def\RR{\mathbb R}
\def\Rn{{\mathbb R}^n}
\def\Rm{{\mathbb R}^m}
\def\Rd{{\mathbb R}^d}

 \medskip

\arraycolsep=1pt

\title{\Large\bf   Lacunary $\delta$-Discretised Spherical Maximal Operators}
\author{Surjeet Singh Choudhary, Ji Li, Chong-Wei Liang, Chun-Yen Shen}

\date{}
\maketitle

\begin{center}
\begin{minipage}{13.5cm}\small

{\noindent  {\bf Abstract:}\  
We study the lacunary analogue of the $\delta$-discretised spherical maximal operators introduced by Hickman and Jan\v{c}ar, for $\delta \in (0, 1/2)$, and establish the boundedness on $L^p$ for all $1 < p < \infty$, along with the endpoint weak-type estimate $H^1 \to L^{1,\infty}$. We also prove the corresponding $L^p$ boundedness for the multi-parameter variant. The constants in these bounds are uniform in $\delta$, and thus, by taking the limit $\delta \to 0^+$, our results recover the classical boundedness of the lacunary spherical maximal function.
}

\end{minipage}
\end{center}



\section{Introduction}
\label{s:1}
\setcounter{equation}{0}

As shown by Duoandikoetxea and Rubio de Francia \cite{DR},  
$L^p(\Rd)$ boundedness of the lacunary maximal function $M_{lac}(f)(x):=\sup_{k\in\mathbb Z} |f*\mu_k(x)|$ holds in the range $1<p<\infty$ if the Fourier transform $\hat \mu$ of the compactly supported finite Borel measure $\mu$ satisfies the decay condition 
$\hat \mu(\xi) = O(|\xi|^{-\epsilon})$ for some $\epsilon>0$, where $\mu_k$ denotes the dyadic dilates defined by $\langle \mu_k,f(\cdot)\rangle = \langle \mu, {f(2^{k}\cdot)}\rangle$.  
The endpoint estimate for \( p = 1 \) has also attracted significant interest. In the Euclidean setting, Christ \cite{Chr} demonstrated that the lacunary spherical maximal function is bounded from \( H^1(\Rd) \) to \( L^{1,\infty}(\Rd) \). 

This line of research has been further developed in various directions. For example, Oberlin \cite{Ober} and Heo \cite{Heo} extended Christ's result to more general measures under a Fourier decay condition on \( \widehat{\mu} \). Seeger and Wright \cite{MR2918096} proved the \( H^1 \to L^{1,\infty} \) boundedness under a specific regularity condition on the convolution operator \( f \mapsto f \ast \mu \). J.~Lee, S.~Lee and Oh \cite{LLO}, as well as Hickman and Zahl \cite{HZ}, further studied the $L^p$ boundedness of the strong spherical maximal operator associated with multi-parameter dilations. See also recent developments on the homogeneous group setting \cite{MR4523247,MR4693933,GHW,GT}.

Recently, Hickman and Jan\v{c}ar \cite{HJ} introduced the $\delta$-discretised spherical maximal operators. 
Given $0<\delta<1/2$, $x\in\Rd$ and $1\leq r\leq2$, let $C(x,r)$ be the sphere in $\Rd$ centred at $x$ with radius $r$, and define
\begin{align*}
C^\delta(x,r) := \{y\in \Rd: |y-z|<\delta,\, z\in C(x,r) \},
\end{align*}
i.e., the $\delta$-neighbourhood of $C(x,r)$, forming an annulus with inner radius $r-\delta$ and outer radius $r+\delta$.
They introduced
\[
M^\delta(f)(x) := \sup_{1\leq r\leq2} \frac{1}{| C^\delta(x,r)| } \int_{C^\delta(x,r)} |f(y)|\,dy,
\qquad x\in\Rd,
\]
and proved that for all $d\geq2$ and $p\geq p_d:= {d\over d-1}$, there exists a constant 
$C_{d,p}\geq1$ such that 
\[
\|M^\delta(f)\|_{L^p(\Rd)}\leq C_{d,p} (\log \delta^{-1})^{p_d/p} \|f\|_{L^p(\Rd)}
\]
holds for all $f\in L^p(\Rd)$ and all $0<\delta<1/2$.
Their proof relies on a geometric approach that avoids the use of the Fourier transform. 

The $\delta$-neighbourhood variants of various maximal operators arise naturally in analysis and geometric measure theory \cite{MR3617376}. We refer the reader to foundational works by Cordoba \cite{Cor}, 
Schlag \cite{Sch1,Sch2}, and Wolff \cite{Wol}, as well as the recent result by Chang, Dosidis, and Kim \cite{CDK}. (see also closely related references  \cite{CMSGCIN, MR4800578}).

Motivated by this direction and the early development of the spherical mean (\cite{MR0420116,MR0537803,MR0874045, MR1448717}), we consider the lacunary version of the $\delta$-discretised spherical maximal operator and its multi-parameter counterpart. 
We establish $L^p$-boundedness results for these operators, and in the one-parameter case, we also prove the endpoint estimate from $H^1$ to $L^{1,\infty}$. 

Importantly, we keep track of the parameter $\delta \in(0,1/2)$ and show that the bounds we obtain are independent of $\delta$. Throughout the paper, unless otherwise stated, all implicit constants in $\lesssim$ are independent of $\delta$. 

We now present our results. Throughout the whole paper, we 
consider $\Rd$ with $d\geq2$ and $\mu_k=\sigma_k$ denotes dyadic dilation of the surface measure on the unit sphere $\mathbb{S}^{d-1}$.

\begin{definition}\label{def1}
Suppose $0<\delta<1/2$. Define the lacunary $\delta$-discretised spherical maximal operators by
\begin{align}
M_{lac}^\delta(f)(x):= \sup_{k\in\mathbb Z}  |f *_\delta \sigma_  k  (x)|,
\qquad x\in\Rd,
\end{align}
where 
\begin{align}
f *_\delta \sigma_  k  (x) = {1\over | C^\delta(0,1) | }
\int_{C^\delta(0,1) } f(x-2^ky) dy.
\end{align}
Here  $2^ky = (2^ky_1,\ldots, 2^ky_d)$.

\end{definition}

We then have the following estimates.
\begin{theorem}\label{thm0}
Suppose $0<\delta<1/2$. Then the following mapping properties hold:

(1) $ M_{lac}^\delta(f)$ is bounded on $L^p(\Rd)$ for all $1<p<\infty$ with
$$ \| M_{lac}^\delta(f)\|_{L^p(\Rd)} \lesssim \|f\|_{L^p(\Rd)}.$$

(2) $ M_{lac}^\delta(f)$ is bounded from 
$H^1(\Rd)$ to $L^{1,\infty}(\Rd)$: for all $\lambda>0$,
\begin{align*}
|\{x\in\Rd: M_{lac}^\delta(f)(x)>\lambda\}|\lesssim {1\over\lambda} \|f\|_{H^1(\Rd)}.
\end{align*}
Here the implicit constants in the inequalities in (1) and (2) are independent of $\delta$.

\end{theorem}

We further study the multi-parameter version of the maximal operator 
as in Definition \ref{def1}, modelled on J. Lee, S. Lee and Oh \cite{LLO} and  Hickman and Zahl \cite{HZ}.
\begin{definition}\label{def2}
Suppose $0<\delta<1/2$. Define the strong lacunary $\delta$-discretised maximal operator 
\begin{align}
\mathcal M_{lac}^\delta(f)(x):= \sup_{\vec{k}\in \mathbb Z^d}  |f *_\delta \sigma_{\vec{k}}  (x)|, \qquad x\in\Rd,
\end{align}
where 
\begin{align}
f *_\delta \sigma_{\vec{k}}  (x) = {1\over | C^\delta(0,1) | }
\int_{C^\delta(0,1) } f(x-2^{\vec{k}}y) dy.
\end{align}
Here  $2^{\vec{k}}y = (2^{k_1}y_1,\ldots, 2^{k_d}y_d)$.
\end{definition}

\begin{theorem}\label{thm1}
Suppose $0<\delta<1/2$. Then $\mathcal M_{lac}^\delta(f)$ is bounded on $L^p(\Rd)$ for all $1<p<\infty$
with
$$ \|\mathcal M_{lac}^\delta(f)\|_{L^p(\Rd)} \lesssim \|f\|_{L^p(\Rd)}.$$
Here the implicit constant in the above inequality is independent of $\delta$.

\end{theorem}


Regarding the approach and techniques, we comment as follows.

\begin{remark}
The implicit constants in Theorems \ref{thm0} and \ref{thm1} are uniform in $\delta$, and thus, by taking the limit $\delta \to 0$, our results recover the classical boundedness of lacunary spherical maximal functions. This can be seen from the standard argument that 
\begin{align*}
M_{lac}(f)(x)\leq \liminf_{\delta\to0^+} M_{lac}^\delta(f)(x),
\qquad x\in\Rd,
\end{align*}
where $M_{lac}(f)(x)$ is the standard lacunary spherical maximal function by Duoandikoetxea and Rubio de Francia \cite{DR} and Christ \cite{Chr}. A similar estimate holds for the lacunary strong spherical maximal function.

\end{remark}

\begin{remark}
In order to keep the implicit constants in Theorem \ref{thm0} (1),  and in Theorem \ref{thm1} to be independent of  $\delta$, for the $L^p$ boundedness, we use 
the Fourier decay estimates, the bootstrapping argument and the (multi-parameter) Littlewood--Paley theory. 
\end{remark}

\begin{remark}
	To prove that the implicit constant in Theorem~\ref{thm0}~(2) for the endpoint
	\( H^1\to L^{1,\infty} \) boundedness is independent of \( \delta \), we first
	establish two different types of Fourier decay estimates for
	\( \varphi_j *_\delta \sigma_k \): one that depends on \( \delta \), and one
	that does not. We also employ a special atomic decomposition of the Hardy
	space due to Seeger and Wright~\cite{MR2918096}.
\end{remark}

\begin{remark}
	Motivated by the ideas of Christ~\cite{Chr} and Seeger and
	Wright~\cite{MR2918096}, we split the lacunary \( \delta \)-discretised
	spherical maximal operator \( M_{lac}^\delta(f) \) into two parts, according
	to the side length of the dyadic cubes (i.e., the supports of the atoms) and
	the thickness of the dilated neighbourhood. For the two resulting parts, we
	use two different stopping time arguments, each combined with the
	corresponding Fourier decay estimate for
	\( \varphi_j *_\delta \sigma_k \).
\end{remark}

\begin{remark}
	The proof of Theorem~\ref{thm0}~(2) is then completed by combining the
	stopping time arguments with the \( L^2 \) boundedness, the Fourier decay
	estimates, almost orthogonality, Littlewood--Paley theory, and the estimate
	on the exceptional set.
\end{remark}

\subsection{Structure of the Article}
This paper is organized as follows. In Section~\ref{s:2}, we present several auxiliary results on Fourier decay and Littlewood--Paley estimates. Section~\ref{s:3} recalls the special atomic decomposition of the Hardy space \( H^1(\mathbb{R}^d) \) due to Seeger and Wright~\cite{MR2918096}. Sections~\ref{s:4} and~\ref{s:5} are devoted to the proofs of Theorems~\ref{thm0} and~\ref{thm1}, respectively.

\section{Preliminaries}\label{s:2}
\setcounter{equation}{0}

Let $\varphi$, $\psi$ be  Schwartz functions in $\G$ having cancellation condition, and let $\varphi_j$ and $\psi_k$  be the dilation as in \eqref{dilation j}.
The following auxiliary lemmas are needed. 

\begin{lemma}\label{lem ortho 2}
     For all $k,j\in\mathbb{Z}$ the following estimates hold:
    \begin{align*}
    (1)~&\|\psi_{k}* \varphi_j\|_{L^1(\Rd)}\lesssim 2^{-|j-k|}.\\
    (2)~&\|g* \varphi_j*_\delta\sigma\|_{L^2(\Rd)}{\lesssim} \min\{1,
    2^{j(d-1)\over2},\delta^{-{1\over2}}2^{\frac{jd}{2}}\}\|g\|_{L^2(\Rd)}.
    \end{align*}
\end{lemma}

\begin{proof}
The argument (1) follows directly from the almost orthogonality estimates
$$|\psi_{k}* \varphi_j(x)|\lesssim 2^{-|j-k|} {2^j+2^k\over (2^j+2^k+|x|)^{d+1}}.$$

We now prove (2). By Plancherel identity, it suffices to estimate $\widehat{\chi_{C^\delta(0,1)}}(\xi)$. Applying the polar coordinate formula, we have
\begin{align*}
\widehat{\chi_{C^\delta(0,1)}}(\xi)
	&:=\int_{\Rd}\chi_{C^\delta(0,1)}(x)e^{-2\pi ix\cdot\xi}dx\\
	&= \int_{1-\delta}^{1+\delta}\widehat{\sigma}(r\xi) r^{d-1}dr\\
    &=2\pi\int_{1-\delta}^{1+\delta}\frac{J_{\frac{d-2}{2}}(2\pi r|\xi|)}{(r|\xi|)^{\frac{d-2}{2}}}r^{d-1}dr\\
    &=2\pi\left[\int_{0}^{1+\delta}\frac{J_{\frac{d-2}{2}}(2\pi r|\xi|)}{(r|\xi|)^{\frac{d-2}{2}}}r^{d-1}dr-\int_{0}^{1-\delta}\frac{J_{\frac{d-2}{2}}(2\pi r|\xi|)}{(r|\xi|)^{\frac{d-2}{2}}}r^{d-1}dr\right].
\end{align*}
By changing variables, we further have
\begin{align*}
&\widehat{\chi_{C^\delta(0,1)}}(\xi)\\
&=2\pi\bigg[\frac{(1+\delta)^{\frac{d}{2}+1}}{|\xi|^\frac{d-2}{2}}\int_{0}^{1}{J_{\frac{d-2}{2}}(2\pi r(1+\delta)|\xi|)}\cdot r^{\frac{d-2}{2}+1 }\,dr\\
&\qquad\qquad-\frac{(1-\delta)^{\frac{d}{2}+1}}{|\xi|^\frac{d-2}{2}}\int_{0}^{1}{J_{\frac{d-2}{2}}(2\pi r(1-\delta)|\xi|)}\cdot r^{\frac{d-2}{2}+1 }\,dr\bigg]\\
&=2\pi\bigg[\frac{(1+\delta)^{\frac{d}{2}+1}}{|\xi|^\frac{d-2}{2}}\frac{J_{\frac{d}{2}}(2\pi(1+\delta)|\xi|)}{2\pi(1+\delta)|\xi|}-\frac{(1-\delta)^{\frac{d}{2}+1}}{|\xi|^\frac{d-2}{2}}\frac{J_{\frac{d}{2}}(2\pi(1-\delta)|\xi|)}{2\pi(1-\delta)|\xi|}\bigg]\\
&=|\xi|^{\frac{-d}{2}}\left[(1+\delta)^{\frac{d}{2}}J_{\frac{d}{2}}(2\pi(1+\delta)|\xi|)-(1-\delta)^{\frac{d}{2}}J_{\frac{d}{2}}(2\pi(1-\delta)|\xi|)\right].
\end{align*}
Using the recursion formula for the Bessel function together with the Mean Value Theorem, we obtain that there exists $-\delta<\theta<\delta$ such that
\begin{align}\label{estimate0.1}
	|\widehat{\chi_{C^\delta(0,1)}}(\xi)|&\leq \frac{2\delta}{|\xi|^{\frac{d}{2}}}\left|\left((\cdot)^{\frac{d}{2}}J_{\frac{d}{2}}(2\pi |\xi|\cdot)\right)'(1+\theta)\right|\notag\\
    &\lesssim \delta\left|\widehat{\sigma}((1+\theta)\xi)\right|\notag\\
    &\leq C\cdot \frac{\delta}{(1+|\xi|)^{\frac{d-1}{2}}}.
\end{align}
This implies that 
\begin{align}\label{estimate1}
	\frac{|\widehat{\chi_{C^\delta(0,1)}}(\xi)|}{|\chi_{C^\delta(0,1)}|}\leq \frac{C}{(1+|\xi|)^{\frac{d-1}{2}}}.
\end{align}
On the other hand, by using the size estimates and the asymptotic expansion of the Bessel function (see for example \cite{MR1232192}), we also have
\begin{align}\label{estimate0.2}
|\widehat{\chi_{C^\delta(0,1)}}(\xi)|&\leq|\xi|^{\frac{-d}{2}}\left[\left|(1+\delta)^{\frac{d}{2}}J_{\frac{d}{2}}(2\pi(1+\delta)|\xi|)\right|+\left|(1-\delta)^{\frac{d}{2}}J_{\frac{d}{2}}(2\pi(1-\delta)|\xi|)\right|\right]\notag\\
	&\leq C\cdot \frac{1}{(1+|\xi|)^{\frac{d+1}{2}}}.
\end{align}
Thus, from the above two estimates \eqref{estimate0.1} and \eqref{estimate0.2}, we get 
\begin{eqnarray}\label{estimate2}
\frac{|\widehat{\chi_{C^\delta(0,1)}}(\xi)|}{|\chi_{C^\delta(0,1)}|}\lesssim {1\over\delta}\left(\frac{1}{(1+|\xi|)^{\frac{d+1}{2}}}\right)^{1\over2}\left(\frac{\delta}{(1+|\xi|)^{\frac{d-1}{2}}}\right)^{1\over2}={1\over\delta^{1\over2}(1+|\xi|)^{d\over2}}.
\end{eqnarray}
Since the Fourier support of $\varphi$ is contained in $ \{\xi: {1\over2}<|\xi|<2\},$
we see that 
$|\xi|\approx 2^{-j}$ 
if $\hat{\varphi}(2^j\xi)\not=0$. Thus, if $j\leq0$,
then by
using the Plancherel theorem and \eqref{estimate2},  we obtain
\begin{align*}
	\|g* \varphi_j*_\delta\sigma\|_{L^2(\Rd)}
	&=
\left\|\widehat{g}(\xi)\widehat{\varphi}(2^j\xi){\widehat{\chi_{C^\delta(0,1)}}\over | C^\delta(0,1) | }(\xi)\right\|_{L^2(\Rd)}\\
&\leq \|\widehat{g}\|_{L^2(\Rd)} \sup_{\xi\in\mathbb{R}^d}\bigg|\widehat{\varphi}(2^j\xi){\widehat{\chi_{C^\delta(0,1)}}\over | C^\delta(0,1) | }(\xi)\bigg|\\
&\lesssim \delta^{-{1\over2}} 2^{jd\over2} \|\widehat{g}\|_{L^2(\Rd)} .
\end{align*}

Using the Plancherel theorem and \eqref{estimate1}, we can also obtain 
\begin{align*}
	\|g* \varphi_j*_\delta\sigma\|_{L^2(\Rd)}
&\lesssim  2^{j(d-1)\over2} \|\widehat{g}\|_{L^2(\Rd)} .
\end{align*}
On the other hand, if $j>0$, then via \eqref{estimate1}, we have
\begin{align*}
	\|g* \varphi_j*_\delta\sigma\|_{L^2(\Rd)}
&\lesssim  \|\widehat{g}\|_{L^2(\Rd)} .
\end{align*}

Combining these three cases, we see that (2) holds. The proof is complete.
\end{proof} 
\smallskip
\begin{lemma}\label{lem ortho 3}
Let $\psi$ and $\varphi$ be two Schwartz functions with cancellations.
     For all $k,j\in\mathbb{Z}$ the following estimates hold:
    \begin{align*}
    |\psi_{k}* \varphi_j*f(x)|\lesssim 2^{-|j-k|} M(f)(x),
    \end{align*}
    where $M(f)(x)$ denotes the standard Hardy--Littlewood maximal function on $\Rd$.
\end{lemma}
\begin{proof}
Via Lemma \ref{lem ortho 2} (1) and the Poisson decay of $\psi_{k}* \varphi_j(x)$, we obtain this pointwise estimate directly.
\end{proof}
\smallskip
\begin{lemma}[Almost Orthogonal Lemma]\label{ort}~\\
    Let $1<p\leq2$, and $\{g_j\}_j$ be the collection of suitable functions such that 
    $$ \sum_j\|g_j\|_{L^p(\Rd)}^p<\infty,$$
    then for every Schwartz function $\varphi$ in $\G$ with high order cancellation and with the dilation $\varphi_j(x)$ defined as in \eqref{dilation j}, we have
   $$\Big\|\sum_j \varphi_j*g_j\Big\|_{L^p(\Rd)}\lesssim \bigg(  \sum_j   \|g_j\|_{L^p(\Rd)}^p
 \bigg)^{\frac{1}{p}}.
   $$ 
\end{lemma}

\begin{proof}
Consider the discrete square function 
$$S_d(f)(x) =\bigg\{\sum_{k} \big|  \psi_k * f(x)\big|^2\bigg\}^{\frac{1}{2}}. $$
By the Littlewood--Paley theory on $\Rd$, we have that 
for all $1<p<\infty$,
$$ \|S_d(f)\|_{L^p(\G)}\simeq \|f\|_{L^p(\G)}, $$
where the implicit constants are independent of $f$. Then we have
\begin{align*}
\bigg\|\sum_j \varphi_j*g_j\bigg\|_{L^p(\Rd)}
&\approx \bigg\| S_d\Big(\sum_j \varphi_j*g_j\Big)\bigg\|_{L^p(\Rd)}
= \bigg\| \bigg\{\sum_{k} \Big| \sum_j \psi_k * \varphi_j*g_j\Big|^2\bigg\}^{\frac{1}{2}}\bigg\|_{L^p(\Rd)}.
\end{align*}

By Lemma \ref{lem ortho 3}, we see that 
\begin{align*}
 \Bigg| \sum_j \psi_k * \varphi_j*g_j(x)\Bigg|& \leq C \sum_j 2^{-|j-k|}   M(g_j)(x)\\
& = C \sum_j 2^{-|j-k|/2}\cdot\left( 2^{-|j-k|/2}  M(g_j)(x)\right).
 \end{align*}
Again,  $M$ denotes the standard Hardy--Littlewood maximal function on $\Rd$.

Therefore, by H\"older's inequality and the almost orthogonality, we have
\begin{align*}
\bigg\|\sum_j \varphi_j*g_j\bigg\|_{L^p(\Rd)}
&\lesssim \bigg\| \bigg\{\sum_{k} 
\bigg|\sum_j2^{-|j-k|}   M(g_j) \bigg|^2
\bigg\}^{\frac{1}{2}}\bigg\|_{L^p(\Rd)}\\
&\lesssim \left\| \left\{\sum_{k} \left(\sum_{j'}2^{-|j'-k|} \cdot  \sum_j2^{-|j-k|}   M(g_j)^2\right)
\right\}^{1\over2}\right\|_{L^p(\Rd)}\\
&\lesssim \bigg\| \bigg\{\sum_j     M(g_j)^2
\bigg\}^{\frac{1}{2} }\bigg\|_{L^p(\Rd)}\\
&\lesssim \bigg\| \bigg\{\sum_j    |g_j|^2
\bigg\}^{\frac{1}{2}}\bigg\|_{L^p(\Rd)}.
\end{align*}
By noting that $p\leq 2$, we further have
\begin{align*}
\bigg\|\sum_j \varphi_j*g_j\bigg\|_{L^p(\Rd)}
&\lesssim  \bigg( \int_{\Rd} \Big\{\sum_j    |g_j(x)|^2
\Big\}^{\frac{p}{2}} dx\bigg)^{\frac{1}{p}}
\leq  \bigg( \int_{\Rd} \sum_j   |g_j(x)|^p
 dx\bigg)^{\frac{1}{p}}\\
&=  \bigg\{  \sum_j   \|g_j\|_{L^p(\Rd)}^p
 \bigg\}^{\frac{1}{p}}.
\end{align*}
The proof is complete.
\end{proof}


\section{Decomposition of Hardy space $H^1(\Rd)$ in Seeger--Wright~\cite{MR2918096}}\label{s:3}
\setcounter{equation}{0}

The real-variable Hardy space theory is one key milestone in the development of harmonic analysis in the 1970s \cite{FS}. Since then, it has been further developed in various settings, see for example \cite{CF, YLK, AE}.

In this section we give a special atomic decomposition of the Hardy space. This result is due to Seeger and Wright~\cite{MR2918096}. For the convenience of the readers and for us to cite the notation and detail in Section \ref{s:4} later, we provide the full details here.

Consider the Peetre type maximal square function
$$
S_{max}(f)(x):=\bigg(\sum_{j\in\mathbb{Z}}\sup_{|x-y| <2^{j}}|\phi_j *f (y)|^2\bigg)^{\frac{1}{2}},
$$
where (and throughout the paper)
\begin{align}\label{dilation j}
\phi_j(x) := 2^{-jd}\phi(2^{-j}x)
\end{align}
and $\phi$ is any non-trivial Schwartz function in $\Rd$ such that 
$${\rm supp}\, \hat\phi\subset \{ \xi\in\Rd: {1\over2}<|\xi|<2\}.$$
\begin{definition}
We define 
$H^1(\Rd)$ to be the completion of $\{f\in L^2(\G): S_{max}(f)\in L^1(\G)\}$
under the norm $\|f\|_{H^1(\Rd)}= \|S_{max}(f)\|_{L^1(\G)}$.
\end{definition}
\begin{remark}
We note that the definition of $H^1(\Rd)$ is equivalent to the standard Hardy space via various characterisations \cite{FS}. 
Note that one characterisation of $H^1(\Rd)$ is via the square function, that is,
 $H^1(\Rd)$ to be the completion of $\{f\in L^2(\G): S(f)\in L^1(\G)\}$
under the norm $\|f\|_{H^1(\Rd)}= \|S(f)\|_{L^1(\G)}$,
where
$$
S(f)(x):=\bigg(\sum_{j\in\mathbb{Z}}|\phi_j *f (x)|^2\bigg)^{\frac{1}{2}}.
$$
It is direct to see that $S(f)(x)\leq S_{max}(f)(x)$, and hence $\|S(f)\|_{L^1(\G)}\leq \|S_{max}(f)\|_{L^1(\G)}$.
It suffices to verify the reverse inequality, that is, $\|S_{max}(f)\|_{L^1(\G)}\lesssim \|S(f)\|_{L^1(\G)}$.
Note also that for $f$ satisfying $\|S(f)\|_{L^1(\G)}<\infty$, one can obtain an atomic decomposition $f=\sum_{j\in\mathbb Z}\lambda_j a_j$, such that each $a_j $ is an atom and that 
$\sum_{j\in\mathbb Z}|\lambda_j| \approx \|S(f)\|_{L^1(\G)}$. Thus, it suffices to show that 
$\|S_{max}(a_j)\|_{L^1(\G)}\lesssim1$ for all $j$. This is a standard estimate via regularity and cancellation.
\end{remark}

We now consider a special atomic decomposition of $H^1(\Rd)$.  
To begin with, we first recall the following representation theorem.

\begin{theorem}[Representation Theorem]\label{rep}~\\
Let $N\in\mathbb{N}$. There exists a function $\varphi\in C^\infty_0(\Rd)$  which has a higher order cancellation up to order $N$ and a function $\phi$ with $\widehat{\phi}\in C^\infty_0(\Rd)$ such that
\begin{align}\label{eq:repro}
f(x)=\sum_{j\in\mathbb{Z}} \varphi_{j}*\varphi_j *  \phi_j*f (x),
\end{align}
where the equality holds in the sense of $L^2(\Rd)$.
\end{theorem}

We let $\mathcal D$ be the system of dyadic cubes in $\G$.
Moreover, let $\mathcal D^j$, $j\in\mathbb Z$, be the subset of $\mathcal D$ such that 
the side-length of $Q$ is
$ 2^{j}$.

\begin{theorem}\label{thm1.2}
    Given $f\in H^1(\Rd)\cap L^2(\G)$, we can decompose 
     \begin{align}\label{atom decomposition}
    f=\sum_{W\in{{\mathcal D}}} b_W,
 \end{align}   
    where this sum converges in $H^1(\Rd)$, each $b_W$ is in $L^2(\G)$ and supported in 20 times enlargement of $W$, and
     \begin{align}\label{atom norm}
   \sum_{W\in{{\mathcal D}}} |W|^{\frac{1}{2}}\cdot \|b_W\|_{L^2(\Rd)}\lesssim\|f\|_{H^1(\Rd)}.
\end{align}
    Moreover, each atom $b_W$ 
    has the following properties.
     \begin{align}\label{atom decomposition further}
b_W=\sum_{j\leq \ell(W)}\varphi_j *\varphi_j *b_{W,j},
 \end{align}   
    where $\varphi_j $ is as in  \eqref{eq:repro},  $\ell(W)$ is the integer such that the side-length of $W$ is $2^{\ell(W)}$, 
    $b_{W,j}$ is supported in $W$ and 
    furthermore, 
    \begin{align}\label{atom norm further}
   \sum_{W\in{{\mathcal D}}} |W|^{\frac{1}{2}}\bigg
(\sum_{j:j\leq \ell(W)} \|b_{W,j}\|^2_{L^2(\G)}\bigg)^{\frac{1}{2}} \lesssim \|f\|_{H^1(\Rd)}.  
    \end{align}
\end{theorem}
\begin{proof}

Suppose $f\in H^1(\Rd)\cap L^2(\G)$. 
By the reproducing formula in Theorem \ref{rep}, we have
\begin{align*}
f(x)=\sum_{
j\in\mathbb{Z}}\varphi_j *\varphi_j *f_j (x)
\end{align*}
with $$f_j =\phi_j *f.$$
Next, we apply Chang and Fefferman's idea \cite{CF} to get the further decomposition. 
For all $\kappa, j\in\mathbb{Z}$, define 
 \begin{eqnarray*}
    \Omega_\kappa&:=&\{x\in \G: S_{max}(f)(x)>2^\kappa \},\\
    B_\kappa&=&\bigcup_j B^{j}_\kappa, \\
    B^{j}_\kappa&:=&\Big\{R\in\D^{j}:\
        |R\cap \Omega_\kappa|\geq{\frac{1}{2}}|R|,\  | R\cap \Omega_{\kappa+1}|< {\frac{1}{2}}|R| \Big\}, {\rm\ and}\\
\widetilde{\Omega}_\kappa&:=&\Big\{x\in \G: M(\chi_{\Omega_\kappa})(x)>{\frac{1}{100^{d}}} \Big\},
\end{eqnarray*}
where $M$ denotes the Hardy--Littlewood maximal function in $\Rd$.

By using Chebyshev's inequality and the $L^2(\G)$-boundedness of $M$, we see that 
\begin{align}\label{tilde omega measure}
|\widetilde{\Omega}_\kappa|\lesssim |\Omega_\kappa|,
\end{align}
where the implicit constant is independent of $\kappa$.

Since $\widetilde{\Omega}_\kappa$ is an open set in $\G$,  by Whitney decomposition, it can be decomposed as the union of maximal dyadic cubes (hence disjoint).
 We denote by $\mathcal W_\kappa$ the collection of all such Whitney cubes associated to $\widetilde{\Omega}_\kappa$, that is,
 $$\mathcal W_\kappa=\{ W \in\D: {\rm maximal\ in\ the\ sense\ that }\ Dil_{{20}}\circ W\subset \widetilde{\Omega}_\kappa\}.$$
 Then we have
 \begin{align}\label{tilde omegak decomp}
\widetilde{\Omega}_\kappa=\bigcup_{W\in \mathcal W_\kappa} W,
 \end{align}

 In what follows, for every dyadic cube $R$ and $W$, we denote $R^*=Dil_{{20}}\circ R$
 and $W^*=Dil_{{20}}\circ W$.
 
 Note that for every dyadic cube $R\in B^j_\kappa$, we have
 $$
 \frac{|R^* \cap\Omega_\kappa|}{|R^*|}\geq  \frac{|R \cap\Omega_\kappa|}{20^{d}|R|}\geq \frac{1}{100^{d}},
 $$
 which implies that $R^*\in\widetilde{\Omega}_\kappa$ and thus there is a unique $W\in \mathcal W_\kappa$ such that $R\subset W.$ Thus, we can further decompose $f_{j}$ by the following
 \begin{align*}
 f_{j}(x)
 &={\sum_{\kappa\in\mathbb{Z}}} \sum_{W\in \mathcal W_\kappa}\sum_{\substack{R\in B^j_{\kappa}\\ R\subset W}} f_{j}(x)\cdot\chi_{R}(x).
 \end{align*}

 We then define
 \begin{align}\label{bWj}
 b_{W,j}(x):=\sum_{\substack{R\in B^j_{\kappa}\\ R\subset W}}f_{j}(x)\cdot\chi_{R}(x)  .
 \end{align}

 Then we further have
 \begin{align}\label{repro e1}
f(x)&=\sum_{
j\in\mathbb{Z}}\varphi_j *\varphi_j *f_j (x)
=\sum_{j\in\mathbb{Z}}\sum_{\kappa\in\mathbb{Z}} \sum_{W\in \mathcal W_\kappa}\varphi_j *\varphi_j *b_{W,j} (x)\notag\\
&=\sum_{\kappa\in\mathbb{Z}}\sum_{W\in \mathcal W_\kappa}\sum_{j:j\leq \ell(W)}\varphi_j *\varphi_j *b_{W,j} (x).
\end{align}
We define 
 \begin{align}\label{bW}
 b_{W}(x)=\sum_{j:j\leq \ell(W)}\varphi_j *\varphi_j *b_{W,j} (x).
 \end{align}
 Then based on \eqref{repro e1}, we have
  \begin{align*}
f(x)=\sum_{\kappa\in\mathbb{Z}}\sum_{W\in \mathcal W_\kappa}b_{W} (x).
\end{align*}
These imply that the atomic decomposition \eqref{atom decomposition} and \eqref{atom decomposition further} hold.

\medskip
We now verify the size estimate \eqref{atom norm further}. To see this,
noting that by using the definition of $b_{W,j}$ in \eqref{bWj} and using the fact that the family of cubes
 $\{R\}_{R\in B^j_{\kappa}, R\subset W}$
are pairwise disjoint, 
we have
    \begin{align}
    &\bigg
(\sum_{j:j\leq \ell(W)} \| b_{W,j} \|^2_{L^2(\G)}\bigg)^{\frac{1}{2}}\nonumber\\
&=\bigg
(\sum_{j:j\leq \ell(W)}\sum_{\substack{R\in B^j_{\kappa}\\ R\subset W}} \|f_{j}(\cdot)\chi_{R}(\cdot)\|^2_{L^2(\G)}\bigg)^{\frac{1}{2}}\nonumber\\
&= 
\bigg(\sum_{j:j\leq \ell(W)}\int_{\bigcup_{\substack{R\in B^j_{\kappa}\\R \subset W}}R}|f_{j}(x)|^2 dx\bigg)^{\frac{1}{2}}\nonumber\\
&\leq 
\bigg(\sum_{j:j\leq \ell(W)}\sum_{\substack{R\in B^j_{\kappa}\\ R\subset W}}|R|\cdot\inf_{x\in R}\sup_{|x-y|<2^{j}}|f_j(y)|^2\bigg)^{\frac{1}{2}}\nonumber\\
&\lesssim 
\bigg(\sum_{j:j\leq \ell(W)} \sum_{\substack{R\in B^j_{\kappa}\\ R\subset W}}|R\backslash\Omega_{\kappa+1}|\cdot\inf_{x\in R\backslash\Omega_{\kappa+1}}\sup_{|x-y|<2^{j}}|f_j(y)|^2\bigg)^{\frac{1}{2}}\nonumber\\
&\lesssim 
\bigg(\sum_{j:j\leq \ell(W)}\sum_{\substack{R\in B^j_{\kappa}\\ R\subset W}}
\int_{R\backslash\Omega_{\kappa+1}}\sup_{|x-y|<2^{j}}|f_j(y)|^2dx\bigg)^{\frac{1}{2}}\nonumber\\
&\lesssim 
\bigg(\sum_{j:j\leq \ell(W)}\int_{W\backslash\Omega_{\kappa+1}}\sup_{|x-y|<2^{j}}|f_j(y)|^2dx\bigg)^{\frac{1}{2}}\nonumber\\
&\lesssim 
\bigg(\int_{W\backslash\Omega_{\kappa+1}}\sum_{j:j\leq \ell(W)}\sup_{|x-y|<2^{j}}|f_j(y)|^2dx\bigg)^{\frac{1}{2}}\nonumber\\
&\leq 
\bigg(\int_{W\backslash\Omega_{\kappa+1}}S_{max}(f)(x)^2 dx\bigg)^{\frac{1}{2}},
\label{equ key1}
\end{align}
where in the second inequality, we use the fact that 
$|R|\lesssim |R\backslash \Omega_{\kappa+1}|$. This follows from the definition of $B^j_{\kappa}$: that is, 
when $R\in B^j_{\kappa}$, $|R\cap\Omega_{\kappa+1}|<{1\over2}|R|$,  which gives 
$|R\backslash \Omega_{\kappa+1}|\geq{1\over2}|R|$.

Then by using \eqref{equ key1}, we further have
\begin{align}
&   \sum_{\kappa\in\mathbb{Z}}\sum_{W\in \mathcal W_\kappa}|W|^{\frac{1}{2}}\bigg
(\sum_{j:j\leq \ell(W)} \| b_{W,j} \|^2_{L^2(\G)}\bigg)^{\frac{1}{2}}\nonumber\\
&\lesssim \sum_{\kappa\in\mathbb{Z}}\sum_{W\in \mathcal W_\kappa}|W|^{\frac{1}{2}}\cdot\bigg(\int_{W\backslash\Omega_{\kappa+1}}S_{max}(f)(x)^2 dx\bigg)^{\frac{1}{2}}\nonumber\\
&\leq \sum_{\kappa\in\mathbb{Z}}\bigg(\sum_{W\in \mathcal W_\kappa}|W|\bigg)^{\frac{1}{2}}\cdot\bigg(\sum_{W\in \mathcal W_\kappa}\int_{W\backslash\Omega_{\kappa+1}}S_{max}(f)(x)^2 dx\bigg)^{\frac{1}{2}},\label{equ key2}
\end{align}
and
the last inequality follows from H\"older's inequality. 

Next, by noting the Whitney decomposition \eqref{tilde omegak decomp} and the measure estimate \eqref{tilde omega measure}, we obtain that the right-hand side of the above inequality \eqref{equ key2} is bounded by 
\begin{align*}
 \sum_{\kappa\in\mathbb{Z}} |\widetilde{\Omega}_\kappa|^{\frac{1}{2}}\bigg(\int_{\widetilde{\Omega}_\kappa\backslash\Omega_{\kappa+1}}S_{max}(f)(x)^2 dx\bigg)^{\frac{1}{2}}
& \leq \sum_{\kappa\in\mathbb{Z}} |\widetilde{\Omega}_\kappa|^{\frac{1}{2}}\cdot 2^{\kappa+1}\cdot |\widetilde{\Omega}_\kappa|^{\frac{1}{2}}
\lesssim \sum_{\kappa\in\mathbb{Z}}2^\kappa |\Omega_\kappa|\\
&\lesssim \|S_{max}f\|_{L^1(\G)}.
\end{align*}

Therefore, \eqref{atom norm further} holds.

\medskip

From the proof of  \eqref{atom norm further}, we can also summarise
that  the following estimate holds
 \begin{align}\label{atom norm gammaW}
\sum_{\kappa\in\mathbb{Z}}\sum_{W\in \mathcal W_\kappa}|W|^{\frac{1}{2}}\cdot \gamma_{W,\kappa}
\lesssim \|f\|_{H^1(\Rd)},
\end{align}
where
 \begin{align}\label{gammaW}
\gamma_{W,\kappa}:=\bigg(\sum_{j:j\leq \ell(W)}\sum_{\substack{R\in B^j_{\kappa}\\ R\subset W}} \|f_{j}(\cdot)\chi_{R}(\cdot)\|^2_{L^2(\G)}\bigg)^{\frac{1}{2}}.
\end{align}
\medskip
Now it suffices to verify \eqref{atom norm}. From the definition of $b_W$ as in \eqref{bW} and by using Lemma \ref{ort}, we have 
\begin{align*}
\|b_W\|_{L^2(\Rd)}&:=\bigg\|\sum_{j:j\leq \ell(W)}\varphi_j *\varphi_j *b_{W,j}\bigg\|_{L^2(\Rd)}\\
&\lesssim\bigg(\sum_{j:j\leq \ell(W)}\Big\|\varphi_j *\Big(\sum_{\substack{R\in B^j_{\kappa}\\ R\subset W}}f_{j}(\cdot)\chi_{R}(\cdot)\Big)\Big\|^2_{L^2(\Rd)}\bigg)^{\frac{1}{2}}\\
&\lesssim\bigg(\sum_{j:j\leq \ell(W)}\Big\|\sum_{\substack{R\in B^j_{\kappa}\\ R\subset W}}f_{j}(\cdot)\chi_{R}(\cdot)\Big\|^2_{L^2(\Rd)}\bigg)^{\frac{1}{2}}\\
&=\bigg(\sum_{j:j\leq \ell(W)}\sum_{\substack{R\in B^j_{\kappa}\\ R\subset W}} \|f_{j}(\cdot)\chi_{R}(\cdot)\|^2_{L^2(\Rd)}\bigg)^{\frac{1}{2}}\\
&=\gamma_{W,\kappa}.
\end{align*}

The above estimate on $\|b_W\|_{L^2(\Rd)}$ implies that
\begin{align*}
\sum_\kappa\sum_{W\in\mathcal W_\kappa}|W|^{\frac{1}{2}}\cdot \|b_W\|_{L^2(\G)}&\lesssim \sum_\kappa\sum_{W\in\mathcal W_\kappa}|W|^{\frac{1}{2}}\cdot \gamma_{W,\kappa}
\lesssim\|f\|_{H
^1(\Rd)},
\end{align*}
where the last inequality follows from \eqref{atom norm gammaW}. Thus, \eqref{atom norm} holds.

The proof of Theorem \ref{thm1.2} is complete.
\end{proof}


\section{The Endpoint Estimate: Proof of $(2)$ in Theorem \ref{thm0}}\label{s:4}
\setcounter{equation}{0}

The statement (1) in Theorem \ref{thm0} is a direct corollary of Theorem \ref{thm1}.
We leave it to the proof of Theorem \ref{thm1}.  We now prove the statement (2) in Theorem \ref{thm0}. It suffices to show that for all $\delta\in(0,1/2)$ and for all $\lambda>0$,
\begin{align}\label{key thm0}
\Big|\big\{x\in \Rd:  \sup_  k  |f *_\delta \sigma_  k  (x)|>\lambda\big\}\Big|\lesssim\frac{1}{\lambda}\|f\|_{H^1(\Rd)}.
\end{align}

For each dyadic cube $W\in\mathcal D$, we write $W=c_W + B_{2^{\ell(W)}}$, where $c_W$ represents the center of $W$ and $B_{2^{\ell(W)}}$ represents the dyadic cube centered at the origin with side-length $2^{\ell(W)}$.

To continue, recall that from Section \ref{s:3}, we have
$$f(x)=\sum_{\kappa\in\mathbb{Z}}\sum_{W\in \mathcal W_\kappa} b_W (x).$$  
  
 Then 
\begin{align}\label{step1 in proof}
&\Big|\big\{x\in \Rd:  \sup_  k  |f *_\delta \sigma_  k  (x)|>\lambda\big\}\Big|\nonumber\\
&\leq \Bigg|\bigg\{x\in \Rd:  \sup_  k  \Big|\sum_{\kappa\in\mathbb{Z}}
\sum_{\substack{W\in \mathcal W_\kappa\\ 2^k\delta\leq 2^{\ell(W)}}} b_W *_\delta \sigma_  k  (x)\Big|>\lambda\bigg\}\Bigg| \nonumber\\
&\qquad+ \Bigg|\bigg\{x\in \Rd:  \sup_  k  \Big|\sum_{\kappa\in\mathbb{Z}}\sum_{\substack{W\in \mathcal W_\kappa\\ 2^k\delta> 2^{\ell(W)}}}  b_W *_\delta \sigma_  k  (x)\Big|>\lambda\bigg\}\Bigg|\nonumber\\
&=:{\bf I}+{\bf II}.
\end{align}

\subsection{Estimate of {\bf I}}    
We aim to prove that for all $\lambda>0$,
  \begin{align}\label{key thm0 1}
 {\bf I} \lesssim\frac{1}{\lambda}\|f\|_{H^1(\Rd)}.
\end{align}
Define $\widetilde{\tau}(W)$ to be the smallest integer such that
\begin{align}\label{mini 1}
2^{\widetilde{\tau}(W)\cdot (d-1)} \cdot 2^{\ell(W)} \geq \frac{1}{\lambda} \cdot|W|^{\frac{1}{2}}\cdot \gamma_{W,\kappa},
\end{align}
where 
$\gamma_{W,\kappa}$ is given in \eqref{gammaW}.

Let $$\tau(W):=\max\{\widetilde{\tau}(W),\, {\ell(W)}\}.$$ 

To prove \eqref{key thm0}, we decompose {\bf I} into three parts:
\begin{align*}
Term_{1,1}(x)&=\sup_{  k \in\mathbb{Z}}\Big|\sum_{\kappa\in\mathbb Z}\  \sum_{\substack{W\in\mathcal W_\kappa\\2^k\delta\leq 2^{\ell(W)}\\ \tau(W)<  k }}b_{W}*_\delta \sigma_  k (x)\Big|\, ,
\\
Term_{1,2}(x)&=\sup_{  k \in\mathbb{Z}}\Big|\sum_{\kappa\in\mathbb Z}\  \sum_{\substack{W\in\mathcal W_\kappa\\2^k\delta\leq 2^{\ell(W)}\\ \ell(W)<  k \leq \tau(W)}}b_{W}*_\delta \sigma_  k (x)\Big|,
\\
Term_{1,3}(x)&=\sup_{  k \in\mathbb{Z}}\Big| \sum_{\kappa\in\mathbb Z}\  \sum_{\substack{W\in\mathcal W_\kappa\\2^k\delta\leq 2^{\ell(W)}  \\ k \leq \ell(W)}}b_{W}*_\delta \sigma_  k (x)\Big|.
\end{align*}

\subsubsection{$Term_{1,1}$: Almost Orthogonality and $L^2$-Boundedness}

Now, we move back to the estimate of $Term_{1,1}$. From the definition, we have  
\begin{align*}
&\|Term_{1,1}\|^2_{L^2(\Rd)}\\
&\leq\bigg\|\sup_{k\in\mathbb{Z}}\Big| \sum_{\kappa\in\mathbb Z}\  \sum_{\substack{W\in\mathcal W_\kappa\\2^k\delta\leq 2^{\ell(W)}\\ \tau(W)<  k }}\sum_{j:j\leq \ell(W)}\varphi_j  *\varphi_j*\Big(\sum_{\substack{R\in B^j_{\kappa}\\ R\subset W}} f_{j}(\cdot)\chi_{R}(\cdot)\Big)*_\delta\sigma_{k}\Big|\bigg\|^2_{L^2(\Rd)}\\ 
&=\bigg\|\sup_{k\in\mathbb{Z}}\Big|\sum_{L\in\mathbb{Z}}\ \ \sum_{\kappa\in\mathbb Z}\sum_{\substack{ W\in\mathcal W_\kappa\\2^k\delta\leq 2^{\ell(W)}\\ \tau(W)<k\\ \ell(W)=L}}\ \sum_{n\geq 0}  \varphi_{L-n}  *\Big(\sum_{\substack{R\in B^{L-n}_{\kappa}\\ R \subset W}} f_{L-n}(\cdot)\chi_{R}(\cdot)\Big)*\varphi_{L-n}*_\delta\sigma_{k}\Big|\bigg\|^2_{L^2(\Rd)}\\ 
&=\Bigg\|\sup_{k\in\mathbb{Z}}\bigg|\sum_{n\geq 0}\bigg
(\sum_{L\in\mathbb{Z}} \ \ \sum_{\kappa\in\mathbb Z}\sum_{\substack{ W\in\mathcal W_\kappa\\2^k\delta\leq 2^{\ell(W)}\\\tau(W)<k\\ \ell(W)=L+n}}\   \varphi_{L}  *\Big(\sum_{\substack{R\in B^{L}_{\kappa}\\ R \subset W}} f_{L}(\cdot)\chi_{R}(\cdot)\Big)*\varphi_{L}*_\delta\sigma_{k}\bigg)\bigg|\Bigg\|^2_{L^2(\Rd)}\\
&=:\bigg\|\sup_{k\in\mathbb{Z}}\Big|\sum_{n\geq 0}I_{n,k}\Big|\bigg\|^2_{L^2(\Rd)},
\end{align*}
where the second equality follows from the substitution and the third one from changing the order of summation. To continue, we further have
\begin{align*}
\|Term_{1,1}\|^2_{L^2(\Rd)}
&\leq\bigg\|\sum_{n\geq 0}\sup_{k\in\mathbb{Z}}\big|I_{n,k}\big|\bigg\|^2_{L^2(\Rd)}\\
&\leq \bigg(\sum_{n\geq 0}\Big\|\sup_{k\in\mathbb{Z}}\big|I_{n,k}\big|\Big\|_{L^2(\Rd)}\bigg)^2\\
&\leq\bigg(\sum_{n\geq 0}\Big(\sum_{k\in\mathbb{Z}}\|I_{n,k}\|^2_{L^2(\Rd)}\Big)^{\frac{1}{2}}\bigg)^2.
\end{align*}

We now estimate $\|I_{n,k}\|_{L^2(\Rd)}$ for each $n$ and $k$. By using Lemma \ref{ort}
 and Lemma \ref{lem ortho 2}, we have
\begin{align*}
&\|I_{n,k}\|_{L^2(\Rd)}=\bigg\|\sum_{L\in\mathbb{Z}}\  \varphi_{L} *\bigg[ \sum_{\kappa\in\mathbb Z}\sum_{\substack{ W\in\mathcal W_\kappa\\2^k\delta\leq 2^{\ell(W)}\\ \tau(W)<k\\ \ell(W)=L+n}}\   \Big(\sum_{\substack{R\in B^{L}_{\kappa}\\R \subset W}} f_{L}(\cdot)\chi_{R}(\cdot)\Big)*\varphi_{L}*_\delta\sigma_{k}\bigg]\ \ \bigg\|_{L^2(\Rd)}\\
&\hskip-.55cm\underset{\text{Lemma}\,\ref{ort}}{\lesssim} \bigg(\sum_{L\in\mathbb{Z}}\bigg\| \ \sum_{\kappa\in\mathbb Z}\sum_{\substack{ W\in\mathcal W_\kappa\\2^k\delta\leq 2^{\ell(W)}\\ \tau(W)<k\\ \ell(W)=L+n}}\  \Big(\sum_{\substack{R\in B^{L}_{\kappa}\\R \subset W}} f_{L}(\cdot)\chi_{R}(\cdot)\Big)*\varphi_{L}*_\delta\sigma_{k} \bigg\|^2_{L^2(\Rd)}\bigg)^{\frac{1}{2}}\\
&\leq \bigg(\sum_{L\in\mathbb{Z}}\|\varphi_{L}*_\delta\sigma_{k}\|^2_{L^2(\Rd) \to L^2(\Rd)}\cdot \bigg\|\ \ \sum_{\kappa\in\mathbb Z} \sum_{\substack{ W\in\mathcal W_\kappa\\2^k\delta\leq 2^{\ell(W)}\\ \tau(W)<k\\ \ell(W)=L+n}}\  \Big(\sum_{\substack{R\in B^{L}_{\kappa}\\R \subset W}} f_{L}(\cdot)\chi_{R}(\cdot)\Big)\bigg\|^2_{L^2(\Rd)}\bigg)^{\frac{1}{2}}\\
&=\bigg(\sum_{L\in\mathbb{Z}}\|\varphi_{L-k}*_\delta\sigma\|^2_{L^2(\Rd) \to L^2(\Rd)}\cdot \bigg\|\ \  \sum_{\kappa\in\mathbb Z}\sum_{\substack{ W\in\mathcal W_\kappa\\2^k\delta\leq 2^{\ell(W)}\\ \tau(W)<k\\ \ell(W)=L+n}}\  \Big(\sum_{\substack{R\in B^{L}_{\kappa}\\R \subset W}} f_{L}(\cdot)\chi_{R}(\cdot)\Big)\bigg\|^2_{L^2(\Rd)}\bigg)^{\frac{1}{2}}\\
&\hskip-.55cm\underset{Lemma\,\ref{lem ortho 2}}{\lesssim} \bigg(\sum_{L\in\mathbb{Z}} \min\{1,  2^{(L-k)(d-1)}\} \cdot\bigg\|\ \ \sum_{\kappa\in\mathbb Z}\sum_{\substack{ W\in\mathcal W_\kappa\\2^k\delta\leq 2^{\ell(W)}\\ \tau(W)<k\\ \ell(W)=L+n}}  \Big(\sum_{\substack{R\in B^{L}_{\kappa}\\R \subset W}} f_{L}(\cdot)\chi_{R}(\cdot)\Big)\bigg\|^2_{L^2(\Rd)}\bigg)^{\frac{1}{2}}\\
&\lesssim  \bigg(\sum_{\substack{L\in\mathbb{Z}\\L<k}} 2^{(L-k)(d-1)}\cdot\bigg\|\ \ \sum_{\kappa\in\mathbb Z}\sum_{\substack{ W\in\mathcal W_\kappa\\2^k\delta\leq 2^{\ell(W)}\\ \tau(W)<k\\ \ell(W)=L+n}}  \Big(\sum_{\substack{R\in B^{L}_{\kappa}\\R \subset W}} f_{L}(\cdot)\chi_{R}(\cdot)\Big)\bigg\|^2_{L^2(\Rd)}\bigg)^{\frac{1}{2}}\\
&\qquad+ \bigg(\sum_{\substack{L\in\mathbb{Z}\\L\geq k}}\bigg\|\ \ \sum_{\kappa\in\mathbb Z}\sum_{\substack{ W\in\mathcal W_\kappa\\2^k\delta\leq 2^{\ell(W)}\\ \tau(W)<k\\ \ell(W)=L+n}}  \Big(\sum_{\substack{R\in B^{L}_{\kappa}\\R \subset W}} f_{L}(\cdot)\chi_{R}(\cdot)\Big)\bigg\|^2_{L^2(\Rd)}\bigg)^{\frac{1}{2}}  .
\end{align*}
We first note that the second term on the right-hand side of the last inequality above does not exist, since it implies $\ell(W)>n+k>n+\tau(W)$, which contradicts the definition of $\tau(W)$.
Now, we substitute this upper bound of $\|I_{n,k}\|_{L^2(\Rd)}$ back to the estimate of $Term_{1,1}$ above. Then it yields that
\begin{align*}
&\bigg(\sum_{k\in\mathbb{Z}}\|I_{n,k}\|^2_{L^2(\Rd)}\bigg)^{\frac{1}{2}}\\
&\lesssim \bigg(\sum_{k\in\mathbb{Z}} \sum_{\substack{L\in\mathbb{Z}\\L<k}} \,\sum_{\kappa\in\mathbb Z}\sum_{\substack{ W\in\mathcal W_\kappa\\2^k\delta\leq 2^{\ell(W)}\\ \tau(W)<k\\ \ell(W)=L+n}}  2^{(L-k)(d-1)}\cdot\bigg\|\sum_{\substack{R\in B^{L}_{\kappa}\\R \subset W}} f_{L}(\cdot)\chi_{R}(\cdot)\bigg\|^2_{L^2(\Rd)}\bigg)^{\frac{1}{2}}
\\
&=2^{-{n(d-1)\over2}}\cdot\bigg( \sum_{\kappa\in\mathbb Z}\sum_{W \in\mathcal W_\kappa}\sum_{\substack{k:2^k\delta\leq 2^{\ell(W)}\\k>\tau(W)\\\ell(W)<n+k}}2^{(\ell(W)-k)(d-1)}\cdot\bigg\|\sum_{\substack{R\in B^{\ell(W)-n}_{\kappa}\\R \subset W}} f_{\ell(W)-n}(\cdot)\chi_{R}(\cdot)\bigg\|^2_{L^2(\Rd)}\bigg)^{\frac{1}{2}}\\
&\lesssim 2^{-{n(d-1)\over2}}\cdot\bigg(\sum_{\kappa\in\mathbb Z}\sum_{W \in\mathcal W_\kappa}2^{(\ell(W)-\tau(W))(d-1)}\cdot \bigg\|\sum_{\substack{R\in B^{\ell(W)-n}_{\kappa}\\R \subset W}} f_{\ell(W)-n}(\cdot)\chi_{R}(\cdot)\bigg\|^2_{L^2(\Rd)}\bigg)^{\frac{1}{2}}\\
&\leq 2^{-{n(d-1)\over2}}\cdot\bigg(\sum_{\kappa\in\mathbb Z}\sum_{W \in\mathcal W_\kappa}2^{(\ell(W)-\tau(W))(d-1)}\cdot \gamma_{W,\kappa}^2\bigg)^{\frac{1}{2}}.
\end{align*}

Recalling the definition of the stopping time parameter $\widetilde{\tau}(W)$:
\begin{align*}
2^{\widetilde{\tau}(W)\cdot (d-1)} \cdot 2^{\ell(W)} \geq \frac{1}{\lambda} \cdot|W|^{\frac{1}{2}}\cdot \gamma_{W,\kappa},
\end{align*}
we have $$\gamma_{W,\kappa}\leq 2^{\widetilde{\tau}(W)\cdot (d-1)} \cdot  2^{\ell(W)}\cdot\lambda \cdot |W|^{-{1\over2}} \leq 2^{\tau(W)\cdot (d-1)} \cdot  2^{\ell(W)} \cdot\lambda \cdot |W|^{-{1\over2}},$$
where we use the fact that $\tau(W)\geq \widetilde{\tau}(W)$.
 
Hence,
\begin{align*}
&\bigg(\sum_{k\in\mathbb{Z}}\|I_{n,k}\|^2_{L^2(\Rd)}\bigg)^{\frac{1}{2}}\\
&\lesssim  2^{-{n(d-1)\over2}}\cdot\bigg(\sum_{\kappa\in\mathbb Z}\sum_{W \in\mathcal W_\kappa}2^{(\ell(W)-\tau(W))(d-1)}\cdot  2^{\tau(W)\cdot (d-1)} \cdot 2^{\ell(W)} \cdot\lambda \cdot |W|^{-{1\over2}} \cdot \gamma_{W,\kappa} \bigg)^{\frac{1}{2}}
\\
&\lesssim   2^{-{n(d-1)\over2}}\cdot\bigg(\sum_{\kappa\in\mathbb Z}\sum_{W \in\mathcal W_\kappa} \lambda  |W|^{1\over2}\gamma_{W,\kappa}\bigg)^{\frac{1}{2}},
 \end{align*}
which leads to
\begin{align}\label{Term11}
\frac{1}{\lambda^2}\|Term_{1,1}\|^2_{L^2(\Rd)}&\lesssim\frac{1}{\lambda^2}\bigg(\sum_{n\geq 0}
2^{-{n(d-1)\over2}}\cdot\bigg(\sum_{\kappa\in\mathbb Z}\sum_{W \in\mathcal W_\kappa} \lambda  |W|^{1\over2}\gamma_{W,\kappa}\bigg)^{\frac{1}{2}}\bigg)^2 \nonumber\\
&\lesssim\frac{1 }{\lambda^2}\cdot\lambda \sum_{\kappa\in\mathbb Z}\sum_{W \in\mathcal W_\kappa}\gamma_{W,\kappa}|W|^{\frac{1}{2}}\nonumber\\
&\lesssim \frac{1}{\lambda}\|f\|_{H^1(\Rd)},
\end{align}
where the last inequality follows from \eqref{atom norm gammaW} in the proof of Theorem \ref{thm1.2}. 

\subsubsection{$Term_{1,2}$: Exceptional Set}

Recall that  $b_{W}$ is supported in 20 times enlargement of $W$, denoted by $W^*$.
Hence, the function
$$
\sup_{  k \in\mathbb{Z}}\Big|\sum_{\kappa\in\mathbb Z}\  \sum_{\substack{W\in\mathcal W_\kappa\\2^k\delta\leq 2^{\ell(W)}\\ \ell(W)<  k \leq \tau(W)}}b_{W}*_\delta \sigma_  k (x)\Big|$$ 
is supported in the set 
$$
\mathcal{E}_{\bf I}:=\bigcup_{\kappa\in\mathbb Z}\  \bigcup_{W\in\mathcal W_\kappa}\bigcup_{\substack{k: \ell(W)<k\leq \tau(W)\\ 2^k\delta\leq 2^{\ell(W)}}}{\left(W^*+_\delta \supp(\sigma_k)\right)},
$$
{in which $W^*+_\delta \supp(\sigma_k)$ is defined to be the sum set of $W^*$ and the $2^k$-dilation of the $\delta$-neighborhood of $\supp(\sigma)$.} Note also that $W^*+_\delta \supp(\sigma_k)$ is contained in the {translation} of the $2^k$-dilation of a $C\cdot 2^{\ell(W)-k}$ neighborhood of $\supp(\sigma)$ for all $\ell(W)<k\leq \tau(W)$, that is
$$
W^*+_\delta \supp(\sigma_k)\subset c_W+Dil_{2^k}\circ\left(\supp(\sigma)+ Dil_C \circ B_{2^{\ell(W)-k}}\right).
$$
Since the Lebesgue measure is translation invariant, then we have
\begin{align*}
  \left|W^*+_\delta \supp(\sigma_k)\right|&\leq \left|Dil_{2^k}\circ\left(\supp(\sigma)+ Dil_C \circ B_{ 2^{\ell(W)-k}}\right)\right|\\
   &\leq 2^{kd}\cdot\left|\supp(\sigma)+ Dil_C \circ B_{ 2^{\ell(W)-k}}\right|.
\end{align*}
For all $r>0$, let $N_{\Sigma(r)}$ be the minimal number of balls of radius $r$ needed to cover the support of $\sigma$ (see also \cite{ MR1333890}). Since the box-packing dimension of a sphere is $d-1$, then it turns out that $N_{\Sigma(r)}\simeq r^{-(d-1)}$. Moreover, $2^k\delta\leq 2^{\ell(W)}$ implies $\delta\leq 2^{\ell(W)-k}$, and hence for all $\ell(W)<k\leq\tau(W)$ with $2^k\delta\leq 2^{\ell(W)}$,
\begin{align*}
  \left|W^*+_\delta \supp(\sigma_k)\right|
  &\lesssim \underset{\text{dilation}}{2^{k\cdot d}}\cdot \underset{\text{number}}{N_{\Sigma(2^{\ell(W)-k})}}\cdot \underset{\text{volume}}{(2^{\ell(W)-k})^d} \\
  &\simeq 2^{k(d-1)}\cdot 2^{\ell(W)} ,
\end{align*}
which implies that
\begin{align*}
\left|\mathcal{E}_{\bf I}\right|
&\lesssim\sum_{\kappa\in\mathbb Z}\  \sum_{W\in\mathcal W_\kappa}\sum_{\substack{k: \ell(W)<k\leq \tau(W)\\ 2^k\delta\leq 2^{\ell(W)}}}  2^{k(d-1)}\cdot 2^{\ell(W)} \\
&\lesssim \sum_{\kappa\in\mathbb Z}\  \sum_{\substack{W\in\mathcal W_\kappa\\ \ell(W)< \tau(W)}} 2^{\tau(W)(d-1)} \cdot 2^{\ell(W)}.
\end{align*}

Then, by the definition of $\tau(W)$, the stopping time argument via \eqref{mini 1} and \eqref{atom norm gammaW}, we can deduce that
\begin{align}\label{eeeeee}
\left|\mathcal{E}_{\bf I}\right|
&\lesssim 2^{d-1}\sum_{\kappa\in\mathbb Z}\  \sum_{\substack{W\in\mathcal W_\kappa\\ \ell(W)< \tau(W)}} 2^{(\tau(W)-1)(d-1)} 2^{\ell(W)}
\notag \\
&\lesssim\sum_{\kappa\in\mathbb Z}\  \sum_{\substack{W\in\mathcal W_\kappa\\ \ell(W)< \tau(W)}}\frac{1}{\lambda} \cdot|W|^{\frac{1}{2}}\cdot \gamma_{W,\kappa}\notag\\
&\lesssim \frac{1}{\lambda} \cdot\|f\|_{H^1(\Rd)}.
\end{align}

\subsubsection{$Term_{1,3}$: $L^2(\Rd)$-Estimate}

Recall that
\begin{align*}
Term_{1,3}(x)
&=\sup_{  k \in\mathbb{Z}}\Big|\sum_{\kappa\in\mathbb Z}\sum_{\substack{W\in \mathcal W_\kappa \\   k \leq \ell(W) \\ 2^k\delta\leq 2^{\ell(W)}}} b_{W}*_\delta\sigma_  k (x)\Big|.
\end{align*}
Since there is a constant $C<\infty$ such that $${\rm supp}\bigg( \sup_{  k :  k \leq \ell(W)}| b_W *_\delta\sigma_  k (x)|\bigg) \subset Dil_C\circ W^*,$$  then by the $L^2(\Rd)$ estimate of the lacunary $\delta$-discretised maximal function, we have
\begin{align}\label{T31}
\|Term_{1,3}\|_{L^1(\Rd)}&:=\bigg\| \sup_{  k \in\mathbb{Z}}\Big|\sum_{\kappa\in\mathbb Z}\sum_{\substack{W\in \mathcal W_\kappa\\   k \leq \ell(W) \\ 2^k\delta\leq 2^{\ell(W)}}} b_{W}*_\delta\sigma_  k (\cdot)\Big|\bigg\|_{L^1(\Rd)}\notag\\
&\leq \sum_{\kappa\in\mathbb Z}\sum_{W\in \mathcal W_\kappa}\bigg\| \sup_{k\leq \ell(W)}\Big| b_{W}*_\delta\sigma_  k (\cdot)\Big|\bigg\|_{L^1(\Rd)}\notag\\
&\leq \sum_{\kappa\in\mathbb Z}\sum_{W\in \mathcal W_\kappa}|W^*|^{\frac{1}{2}}\cdot \big\|\sup_  k | b_W*_\delta\sigma_  k |\big\|_{L^2(\Rd)}\notag\\
&\lesssim \sum_{\kappa\in\mathbb Z}\sum_{W\in \mathcal W_\kappa}|W|^{\frac{1}{2}}\cdot \|b_W\|_{L^2(\Rd)}\nonumber\\
&\lesssim \|f\|_{H^1(\Rd)},
\end{align}
where in the last inequality we use \eqref{atom norm}, and in the last but second we use (1) in Theorem \ref{thm0}.

Combining all the estimates above for $Term_{1,j}$, $j=1,2,3$, (\ref{Term11}), (\ref{eeeeee}) and (\ref{T31}), we obtain that
\begin{align*}
 {\bf I}&\lesssim   \frac{1}{\lambda^2}\|Term_{1,1}\|^2_{L^2(\Rd)}+\left|\mathcal E_{\bf I}\right| +\frac{1}{\lambda}\|Term_{1,3}\|_{L^1(\Rd)}\lesssim   {1\over \lambda} \|f\|_{H^1(\Rd)},
\end{align*}
that is, \eqref{key thm0 1} holds.

\subsection{Estimate of {\bf II}}

We now consider {\bf II} as in \eqref{step1 in proof}.  
We aim to prove that 
  \begin{align}\label{key thm0 2}
 {\bf II} \lesssim\frac{1}{\lambda}\|f\|_{H^1(\Rd)}.
\end{align}

Define $\widetilde{\tau}(W)$ to be the smallest integer such that
\begin{align}\label{mini}
2^{\widetilde{\tau}(W)\cdot d} \cdot {\delta} \geq \frac{1}{\lambda} \cdot|W|^{\frac{1}{2}}\cdot \gamma_{W,\kappa},
\end{align}
where 
$\gamma_{W,\kappa}$ is given in \eqref{gammaW}.
Let $$\tau(W):=\max\Big\{\widetilde{\tau}(W),\,{\ell(W)
+\log_2\Big({1\over\delta}\Big)}\Big\}.$$

To prove \eqref{key thm0 2}, since $ 2^k\delta> 2^{\ell(W)}$, we decompose the 
the term in {\bf II}
 into two parts:
\begin{align*}
Term_{2,1}(x)&=\sup_{  k \in\mathbb{Z}}\Big|\sum_{\kappa\in\mathbb Z}\  \sum_{\substack{W\in\mathcal W_\kappa\\ \tau(W)<  k }}b_{W}*_\delta \sigma_  k (x)\Big|\, ,
\\
Term_{2,2}(x)&=\sup_{  k \in\mathbb{Z}}\Big|\sum_{\kappa\in\mathbb Z}\  \sum_{\substack{W\in\mathcal W_\kappa\\ \ell(W)+\log_2({1\over\delta})<  k \leq \tau(W)}}b_{W}*_\delta \sigma_  k (x)\Big|.
\end{align*}

Thus, for all $\lambda>0$, it suffices to estimate the level set of each terms, that is
\begin{align*}
\Big|\big\{x\in \Rd: Term_{2,j}(x)>\lambda\big\}\Big|,\qquad j=1,2.
\end{align*}

\subsubsection{$Term_{2,1}$: Almost Orthogonality and $L^2$-Boundedness}
Now, we consider $Term_{2,1}$. 
Following the same procedure as in $Term_{1,1}$, we first have
\begin{align*}
&\|Term_{2,1}\|^2_{L^2(\Rd)}\leq\bigg(\sum_{n\geq 0}\Big(\sum_{k\in\mathbb{Z}}\|I_{n,k}\|^2_{L^2(\Rd)}\Big)^{\frac{1}{2}}\bigg)^2,
\end{align*}
in which 
$$
I_{n,k}=\sum_{L\in\mathbb{Z}}\  \varphi_{L} *\bigg[ \sum_{\kappa\in\mathbb Z}\sum_{\substack{ W\in\mathcal W_\kappa\\ \tau(W)<k\\ \ell(W)=L+n}}\   \Big(\sum_{\substack{R\in B^{L}_{\kappa}\\R \subset W}} f_{L}(\cdot)\chi_{R}(\cdot)\Big)*\varphi_{L-n}*_\delta\sigma_{k}\bigg].
$$
We now estimate $\|I_{n,k}\|_{L^2(\Rd)}$ for each $n$ and $k$. By using Lemma \ref{ort}
 and Lemma \ref{lem ortho 2}, we have
\begin{align*}
&\|I_{n,k}\|_{L^2(\Rd)}
=\bigg\|\sum_{L\in\mathbb{Z}}\  \varphi_{L} *\bigg[ \sum_{\kappa\in\mathbb Z}\sum_{\substack{ W\in\mathcal W_\kappa\\ \tau(W)<k\\ \ell(W)=L+n}}\   \Big(\sum_{\substack{R\in B^{L}_{\kappa}\\R \subset W}} f_{L}(\cdot)\chi_{R}(\cdot)\Big)*\varphi_{L}*_\delta\sigma_{k}\bigg]\ \ \bigg\|_{L^2(\Rd)}\\
&\hskip-.55cm\underset{\text{Lemma}\,\ref{ort}}{\lesssim} \bigg(\sum_{L\in\mathbb{Z}}\bigg\| \ \sum_{\kappa\in\mathbb Z}\sum_{\substack{ W\in\mathcal W_\kappa\\ \tau(W)<k\\ \ell(W)=L+n}}\  \Big(\sum_{\substack{R\in B^{L}_{\kappa}\\ R \subset W}} f_{L}(\cdot)\chi_{R}(\cdot)\Big)*\varphi_{L}*_\delta\sigma_{k} \bigg\|^2_{L^2(\Rd)}\bigg)^{\frac{1}{2}}\\
&\leq \bigg(\sum_{L\in\mathbb{Z}}\|\varphi_{L}*_\delta\sigma_{k}\|^2_{L^2(\Rd) \to L^2(\Rd)}\cdot \bigg\|\ \ \sum_{\kappa\in\mathbb Z} \sum_{\substack{ W\in\mathcal W_\kappa\\\tau(W)<k\\ \ell(W)=L+n}}\  \Big(\sum_{\substack{R\in B^{L}_{\kappa}\\ R \subset W}} f_{L}(\cdot)\chi_{R}(\cdot)\Big)\bigg\|^2_{L^2(\Rd)}\bigg)^{\frac{1}{2}}\\
&=\bigg(\sum_{L\in\mathbb{Z}}\|\varphi_{L-k}*_\delta\sigma\|^2_{L^2(\Rd) \to L^2(\Rd)}\cdot \bigg\|\ \  \sum_{\kappa\in\mathbb Z}\sum_{\substack{ W\in\mathcal W_\kappa\\\tau(W)<k\\ \ell(W)=L+n}}\  \Big(\sum_{\substack{R\in B^{L}_{\kappa}\\ R \subset W}} f_{L}(\cdot)\chi_{R}(\cdot)\Big)\bigg\|^2_{L^2(\Rd)}\bigg)^{\frac{1}{2}}\\
&\lesssim  \bigg(\sum_{L\in\mathbb{Z}} \min\{1,\, {\delta^{-1}} 2^{(L-k)d}\} \cdot\bigg\|\ \ \sum_{\kappa\in\mathbb Z}\sum_{\substack{ W\in\mathcal W_\kappa\\ \tau(W)<k\\ \ell(W)=L+n}}  \Big(\sum_{\substack{R\in B^{L}_{\kappa}\\ R \subset W}} f_{L}(\cdot)\chi_{R}(\cdot)\Big)\bigg\|^2_{L^2(\Rd)}\bigg)^{\frac{1}{2}}\\
&\leq {\delta^{-{1\over2}}} \bigg(\sum_{\substack{L\in\mathbb{Z}\\L<k}} 2^{(L-k)d}\cdot\bigg\|\ \ \sum_{\kappa\in\mathbb Z}\sum_{\substack{ W\in\mathcal W_\kappa\\ \tau(W)<k\\ \ell(W)=L+n}}  \Big(\sum_{\substack{R\in B^{L}_{\kappa}\\ R \subset W}}f_{L}(\cdot)\chi_{R}(\cdot)\Big)\bigg\|^2_{L^2(\Rd)}\bigg)^{\frac{1}{2}}\\
&\qquad+ \bigg(\sum_{\substack{L\in\mathbb{Z}\\L\geq k}}\bigg\|\ \ \sum_{\kappa\in\mathbb Z}\sum_{\substack{ W\in\mathcal W_\kappa\\ \tau(W)<k\\ \ell(W)=L+n}}  \Big(\sum_{\substack{R\in B^{L}_{\kappa}\\ R \subset W}} f_{L}(\cdot)\chi_{R}(\cdot)\Big)\bigg\|^2_{L^2(\Rd)}\bigg)^{\frac{1}{2}}\\
&\leq {\delta^{-{1\over2}}} \bigg(\sum_{\substack{L\in\mathbb{Z}\\ L<k}} 2^{(L-k)d}\cdot\bigg\|\ \ \sum_{\kappa\in\mathbb Z}\sum_{\substack{ W\in\mathcal W_\kappa\\ \tau(W)<k\\ \ell(W)=L+n}}  \Big(\sum_{\substack{R\in B^{L}_{\kappa}\\ R \subset W}} f_{L}(\cdot)\chi_{R}(\cdot)\Big)\bigg\|^2_{L^2(\Rd)}\bigg)^{\frac{1}{2}}\\
&\qquad+ \bigg(\sum_{\substack{L\in\mathbb{Z}\\ L\geq k}}\bigg\|\ \ \sum_{\kappa\in\mathbb Z}\sum_{\substack{ W\in\mathcal W_\kappa\\ \tau(W)<k\\ \ell(W)=L+n}}  \Big(\sum_{\substack{R\in B^{L}_{\kappa}\\ R \subset W}} f_{L}(\cdot)\chi_{R}(\cdot)\Big)\bigg\|^2_{L^2(\Rd)}\bigg)^{\frac{1}{2}}\\
&\lesssim {\delta^{-{1\over2}}} \bigg(\sum_{\substack{L\in\mathbb{Z}\\ L<k}} 2^{(L-k)d}\cdot\bigg\|\ \ \sum_{\kappa\in\mathbb Z}\sum_{\substack{ W\in\mathcal W_\kappa\\ \tau(W)<k\\ \ell(W)=L+n}}  \Big(\sum_{\substack{R\in B^{L}_{\kappa}\\ R \subset W}} f_{L}(\cdot)\chi_{R}(\cdot)\Big)\bigg\|^2_{L^2(\Rd)}\bigg)^{\frac{1}{2}}\\
&\qquad+ \bigg(\sum_{\substack{L\in\mathbb{Z}\\ L\geq k}}\bigg\|\ \ \sum_{\kappa\in\mathbb Z}\sum_{\substack{ W\in\mathcal W_\kappa\\ \tau(W)<k\\ \ell(W)=L+n}}  \Big(\sum_{\substack{R\in B^{L}_{\kappa}\\ R \subset W}} f_{L}(\cdot)\chi_{R}(\cdot)\Big)\bigg\|^2_{L^2(\Rd)}\bigg)^{\frac{1}{2}}.
\end{align*}
We first note that the second term on the right-hand side of the last inequality above does not exist, since it implies that  $\ell(W)>n+k>n+\tau(W)$, which contradicts the definition of $\tau(W)$. We now substitute this upper bound of $\|I_{n,k}\|_{L^2(\Rd)}$ back to the estimate of $Term_{2,1}$ above. Then it yields that
\begin{align*}
&\bigg(\sum_{k\in\mathbb{Z}}\|I_{n,k}\|^2_{L^2(\Rd)}\bigg)^{\frac{1}{2}}\\
&\lesssim {\delta^{-{1\over2}}}\bigg(\sum_{k\in\mathbb{Z}} \sum_{\substack{L\in\mathbb{Z}\\ L<k}}\ \ \sum_{\kappa\in\mathbb Z}\sum_{\substack{ W\in\mathcal W_\kappa\\ \tau(W)<k\\ \ell(W)=L+n}}  2^{(L-k)d}\cdot\bigg\|\sum_{\substack{R\in B^{L}_{\kappa}\\ R \subset W}}f_{L}(\cdot)\chi_{R}(\cdot)\bigg\|^2_{L^2(\Rd)}\bigg)^{\frac{1}{2}}
\\
&={\delta^{-{1\over2}}}2^{-{nd\over2}}\cdot\bigg( \sum_{\kappa\in\mathbb Z}\sum_{W \in\mathcal W_\kappa}\sum_{\substack{k:k>\tau(W)\\\ell(W)<n+k}}2^{(\ell(W)-k)d}\cdot\bigg\|\sum_{\substack{R\in B^{\ell(W)-n}_{\kappa}\\R \subset W}} f_{\ell(W)-n}(\cdot)\chi_{R}(\cdot)\bigg\|^2_{L^2(\Rd)}\bigg)^{\frac{1}{2}}\\
&\lesssim {\delta^{-{1\over2}}}2^{-{nd\over2}}\cdot\bigg(\sum_{\kappa\in\mathbb Z}\sum_{W \in\mathcal W_\kappa}2^{(\ell(W)-\tau(W))d}\cdot \bigg\|\sum_{\substack{R\in B^{\ell(W)-n}_{\kappa}\\R \subset W}} f_{\ell(W)-n}(\cdot)\chi_{R}(\cdot)\bigg\|^2_{L^2(\Rd)}\bigg)^{\frac{1}{2}}\\
&\lesssim {\delta^{-{1\over2}}} 2^{-{nd\over2}}\cdot\bigg(\sum_{\kappa\in\mathbb Z}\sum_{W \in\mathcal W_\kappa}2^{(\ell(W)-\tau(W))d}\cdot \gamma_{W,\kappa}^2\bigg)^{\frac{1}{2}}.
\end{align*}
Since $\tau(W)\geq \widetilde{\tau}(W)$ and
\begin{align*}
2^{\widetilde{\tau}(W)\cdot d} \cdot {\delta} \geq \frac{1}{\lambda} \cdot|W|^{\frac{1}{2}}\cdot \gamma_{W,\kappa},
\end{align*}
then we have $$\gamma_{W,\kappa}\leq 2^{\widetilde{\tau}(W)\cdot d} \cdot {\delta} \cdot\lambda \cdot |W|^{-{1\over2}} \leq 2^{\tau(W)\cdot d} \cdot {\delta} \cdot\lambda \cdot |W|^{-{1\over2}}  .$$ 
Hence, one has
\begin{align*}
&\bigg(\sum_{k\in\mathbb{Z}}\|I_{n,k}\|^2_{L^2(\Rd)}\bigg)^{\frac{1}{2}}\\
&\lesssim {\delta^{-{1\over2}}} 2^{-{nd\over2}}\cdot\bigg(\sum_{\kappa\in\mathbb Z}\sum_{W \in\mathcal W_\kappa}2^{(\ell(W)-\tau(W))d}\cdot  2^{\tau(W)\cdot d} \cdot {\delta} \cdot\lambda \cdot |W|^{-{1\over2}} \cdot \gamma_{W,\kappa} \bigg)^{\frac{1}{2}}
\\\
&\lesssim   2^{-{nd\over2}}\cdot\bigg(\sum_{\kappa\in\mathbb Z}\sum_{W \in\mathcal W_\kappa} \lambda  |W|^{1\over2}\gamma_{W,\kappa}\bigg)^{\frac{1}{2}},
 \end{align*}

which leads to the desired estimate
\begin{align}\label{Term21}
\frac{1}{\lambda^2}\|Term_{2,1}\|^2_{L^2(\Rd)}&\lesssim\frac{1}{\lambda^2}\bigg(\sum_{n\geq 0}
2^{-{nd\over2}}\cdot\bigg(\sum_{\kappa\in\mathbb Z}\sum_{W \in\mathcal W_\kappa} \lambda  |W|^{1\over2}\gamma_{W,\kappa}\bigg)^{\frac{1}{2}}\bigg)^2 \nonumber\\
&=\frac{1 }{\lambda^2}\cdot\lambda \sum_{\kappa\in\mathbb Z}\sum_{W \in\mathcal W_\kappa}\gamma_{W,\kappa}|W|^{\frac{1}{2}}\nonumber\\
&\lesssim \frac{1}{\lambda}\|f\|_{H^1(\Rd)},
\end{align}
where the last inequality follows from \eqref{atom norm gammaW} in the proof of Theorem \ref{thm1.2}. 

\subsubsection{$Term_{2,2}$: Exceptional Set}

We consider 
\begin{align}\label{T2 level set}
&\Big|\big\{x\in \Rd: Term_{2,2}(x)>\lambda\big\}\Big|\nonumber\\
&=\bigg|\bigg\{x\in \Rd: \sup_{  k \in\mathbb{Z}}\Big|\sum_{\kappa\in\mathbb Z}\  \sum_{\substack{W\in\mathcal W_\kappa\\ \ell(W)+\log_2({1\over\delta})<  k \leq \tau(W)}}b_{W}*_\delta \sigma_  k (x)\Big|>\lambda\bigg\}\bigg|.
\end{align}
Recall that  $b_{W}$ is supported in 20 times enlargement of $W$, denoted by $W^*$.
Hence, the function
$$\sup_{k\in\mathbb{Z}}\Big|\sum_{\kappa\in\mathbb Z}\  \sum_{\substack{W\in\mathcal W_\kappa\\  \ell(W)+\log_2({1\over\delta})<k\leq \tau(W)}}b_{W}*_\delta\sigma_{k}(x)\Big|$$ 
is supported in the set 
$$
\mathcal{E}_{\bf II}:=\bigcup_{\kappa\in\mathbb Z}\  \bigcup_{W\in\mathcal W_\kappa}\bigcup_{k: \ell(W)+\log_2({1\over\delta})<k\leq\tau(W)}\left(W^*+_\delta \supp(\sigma_k)\right)
$$
and also for all $k$ with $\ell(W)+\log_2({1\over\delta})<k\leq\tau(W)$
\begin{align*}
  \left|W^*+_\delta \supp(\sigma_k)\right|
 \lesssim 2^{kd}\cdot \delta ,
\end{align*}
Then we have
\begin{align*}
\left|\mathcal{E}_{\bf II}\right|
&\lesssim\sum_{\kappa\in\mathbb Z}\  \sum_{W\in\mathcal W_\kappa}\sum_{k:  \ell(W)+\log_2({1\over\delta})< k\leq\tau(W)} 2^{kd}\cdot \delta \\
&\lesssim 2^d\cdot\sum_{\kappa\in\mathbb Z}\  \sum_{\substack{W\in\mathcal W_\kappa\\ \ell(W)+\log_2(\frac{1}{\delta})< \tau(W)}} 2^{(\tau(W)-1)d}\cdot \delta.
\end{align*}
Again, by the definition of $\tau(W)$, (stopping time argument via \eqref{mini}) and \eqref{atom norm gammaW}, we  deduce that
\begin{align*}
\left|\mathcal{E}_{\bf II}\right|
&\lesssim \sum_{\kappa\in\mathbb Z}\  \sum_{\substack{W\in\mathcal W_\kappa\\ \ell(W)< \tau(W)}}\frac{1}{\lambda} \cdot|W|^{\frac{1}{2}}\cdot \gamma_{W,\kappa}\\
&\lesssim \frac{1}{\lambda} \cdot\|f\|_{H^1(\Rd)}.
\end{align*}
Thus, based on \eqref{T2 level set} and the above estimates, we have
\begin{align*}
\Big|\big\{x\in \Rd: Term_{2,2}(x)>\lambda\big\}\Big| \lesssim \frac{1}{\lambda} \cdot\|f\|_{H^1(\Rd)}.
\end{align*}

\bigskip

Combining all the estimates above for $Term_{2,j}$, $j=1,2$,  we obtain that
\begin{align*}
 {\bf II}&\lesssim   \frac{1}{\lambda^2}\|Term_{2,1}\|^2_{L^2(\Rd)}+\left|\mathcal E_{\bf II}\right| \lesssim   {1\over \lambda} \|f\|_{H^1(\Rd)}.
\end{align*}
Hence, \eqref{key thm0 2} holds.

\subsection{Back to the Estimate of \eqref{key thm0}}

Combining \eqref{step1 in proof} and all the estimates in \eqref{key thm0 1} and \eqref{key thm0 2},  we obtain that
\begin{align*}
 |\{x\in \Rd: M_{lac}^\delta(f)(x)>\lambda\}|
  &\lesssim   {1\over \lambda} \|f\|_{H^1(\Rd)}.
\end{align*}
The proof of $(2)$ in Theorem \ref{thm0} is complete.

\section{The $L^p$ $(1<p<\infty)$ Estimate:  Proof of Theorem \ref{thm1}}\label{s:5}
\setcounter{equation}{0}

To show theorem \ref{thm1}, we first introduce the multi-scale decomposition that will be used.

{\it Multi-scale Decomposition}: Let $\psi,\Psi\in\mathcal{S}(\Rd)$ be  Schwartz functions with the Fourier support  \begin{align*}
\supp(\widehat{\psi})\subset\left\{\xi\in\Rd:\,|\xi|\leq2\right\}\quad\text{and}\quad\supp(\widehat{\Psi})\subset\left\{\xi\in\Rd:\,\frac{1}{2}\leq|\xi|\leq2\right\}
\end{align*} 
such that 
\begin{align}\label{identity}
\widehat{\psi}(\xi)+\sum_{j=1}^\infty\widehat{\Psi_{-j}}(\xi)=1,\quad \forall \,\xi\neq0,
\end{align}
where the dilation is defined as in (\ref{dilation j}). 
For a vector $\vec{j}=(j_1,\dots,j_d)\in\mathbb{Z}^d$ and $y=(y_1,\dots,y_d)\in\Rd$, we define
\[\widehat{\psi}_{\vec{j}}(y):=\widehat{\psi}(2^{j_1}y_1)\cdots\widehat{\psi}(2^{j_d}y_d)\quad\text{and}\quad\widehat{\Psi}_{\vec{j}}(y):=\widehat{\Psi}(2^{j_1}y_1)\cdots\widehat{\Psi}(2^{j_d}y_d).\]
Using the identity \eqref{identity}, we can write
\begin{align*}
	f *_\delta \sigma_{\vec{k}} = \sum_{\substack{\vec{j}\in\mathbb{Z}^d:\\j_1,\dots,j_d\geq0}}A_{\vec{j}}^{\vec{k}}f(x),
\end{align*}
where
$$A_{\vec{0}}^{\vec{k}}f(x)=f *\psi_{\vec{k}}*_\delta\sigma_{\vec{k}}(x)\quad\text{and}\quad \A_{\vec{j}}^{\vec{k}}f(x)=f *\Psi_{\vec{k}-\vec{j}}*_\delta \sigma_{\vec{k}}(x).$$

To prove the $L^p$-bound of the $\delta$-discretised Lacunary maximal function, it suffices to study the boundedness of the following operator
$M_{\vec{j}}f(x):=\sup_{\vec{k}\in \mathbb Z^d}|A_{\vec{j}}^{\vec{k}}f(x)|$.
Besides, by Young's inequality and the Poisson bound of the Schwartz function $\Psi$ and the interpolation theorem, it is direct to see that $A_{\vec{j}}^{\vec{k}}$ maps $L^p(\Rd)$ to $L^p(\Rd)$ for all $p\geq1$ and the $L^p$-norm of $A_{\vec{j}}^{\vec{k}}$ is independent of $\delta$. We also have $M_{\vec{0}}f(x)\lesssim M_{st}f(x)$ (see for example section $3$ in \cite{LLO} ),
where $M_{st}$ is the standard strong maximal function, given by 
$$ M_{st}(f)(x)= \sup_{R\ni x} {1\over |R|}\int_{R} |f(y)|dy $$
with $R$ running through all rectangles in $\Rd$ whose sides are parallel to the coordinates. The $L^p$-mapping property of the strong maximal function $M_{st}$ gives that  $M_{\vec{0}}$ is bounded on $L^p(\Rd)$ for all $p>1$.

For a non-trivial vector $\vec{j}$, we will first prove an $L^2$-estimate of $M_{\vec{j}}$ with decay and use a bootstrapping argument to complete the proof of Theorem \ref{thm1}.
\begin{lemma}[$L^2$-Estimate]\label{l2}~\\
	Let $\vec{j}$ be a vector such that each component is non-negative such that $\vec{j}\neq\vec{0}$, then
	\[\|M_{\vec{j}}f\|_{L^2(\Rd)}\lesssim2^{-\max\{j_1,\dots,j_d\}{{d-1}\over2}}\|f\|_{L^2(\Rd)}.\]
\end{lemma}
\begin{proof}
	By the $\ell_2\hookrightarrow\ell_{\infty}$ embedding, we have
$$M_{\vec{j}}f(x)\leq\left(\sum_{\vec{k}\in\mathbb{Z}^d}|A_{\vec{j}}^{\vec{k}}f(x)|^2\right)^{1\over2},$$
which leads to
\begin{align}\label{square}
\|M_{\vec{j}}f\|_{L^2(\Rd)}^2\leq\sum_{\vec{k}\in\mathbb{Z}^d}\|A_{\vec{j}}^{\vec{k}}f\|_{L^2(\Rd)}^2.
\end{align}
	Due to dilation invariance, it is enough to estimate for $A_{\vec{j}}^{\vec{0}}$. Let $\Sigma$ be the measure induced by $*_\delta\sigma$. Using the Plancherel theorem, the support condition of $\widehat{\Psi_{-\vec{j}}}$ and the Fourier decay of measure $\Sigma$ (\ref{estimate1}), we get that
\begin{align*}
\|A_{\vec{j}}^{\vec{0}}f\|_{L^2(\Rd)}=\left\|\widehat{\left(f*\Psi_{-\vec{j}}\right)}(\xi)\cdot\widehat{\Sigma}(\xi)\right\|_{L^2(\Rd)}&\lesssim{2^{-\max\{j_1,\dots,j_d\}{{d-1}\over2}}\|f*\Psi_{-\vec{j}}\|_{L^2(\Rd)}}.
\end{align*}

Therefore, (\ref{square}) leads to
	\begin{align*}
		\|M_{\vec{j}}f\|_{L^2(\Rd)}&\lesssim\left(\sum_{\vec{k}\in\mathbb{Z}^d}{2^{-\max\{j_1,\dots,j_d\}{(d-1)}}}\|f*\Psi_{\vec{k}-\vec{j}}\|_{L^2(\Rd)}^2\right)^{1\over2}\\
		&\lesssim {2^{-\max\{j_1,\dots,j_d\}{{d-1}\over2}}}\|f\|_{L^2(\Rd)},
	\end{align*}
	where we used the multi-parameter Littlewood--Paley inequality in the last inequality.
\end{proof}
Now, we will show the bootstrapping argument. Consider the following vector-valued operator $\vec{\bold{A}}_{\vec{j}}$ acting on a sequence of measurable functions $\bold{f}=\{f*\Psi_{\vec{k}-\vec{j}}\}_{\vec{k}\in\mathbb{Z}^d}$,
\[\vec{\bold{A}}_{\vec{j}}\bold{f}=\{A_{\vec{j}}^{\vec{k}}(f*\Psi_{\vec{k}-\vec{j}})\}_{\vec{k}\in\mathbb{Z}^d}.\]
We can extend the $L^1$-boundedness of $A_{\vec{j}}^{\vec{k}}$ to the vector-valued setting
\[\|\vec{\bold{A}}_{\vec{j}}\|_{L^1(\ell_1)\to L^1(\ell_1)}\lesssim1.\]
Next, using the $L^2$-boundedness of $M_{\vec{j}}$, we get
\begin{align*}
	\|\vec{\bold{A}}_{\vec{j}}\bold{f}\|_{L^2(\ell_\infty)}
	&= \left\|\sup_{\vec{k}\in\mathbb{Z}^d}|A_{\vec{j}}^{\vec{k}}(\bold{f})|\right\|_{L^2}\\&\leq \left\|M_{\vec{j}}(\sup_{\vec{m}}|f*\Psi_{\vec{m}-\vec{j}}|)\right\|_{L^2}\\
	&\lesssim {2^{-\max\{j_1,\dots,j_d\}{{d-1}\over2}}}\|\bold{f}\|_{L^2(\ell_\infty)}.
\end{align*}
Thus, interpolating ${L^1(\ell_1)}$ and ${L^2(\ell_\infty)}$ bounds of $\vec{\bold{A}}_{\vec{j}}$, we get that
\[\|\vec{\bold{A}}_{\vec{j}}\bold{f}\|_{L^{4\over3}(\ell_2)}\lesssim{2^{-\max\{j_1,\dots,j_d\}{{d-1}\over4}}}\|\bold{f}\|_{L^{4\over3}(\ell_2)}.\]
Again using the $\ell_2\hookrightarrow\ell_{\infty}$ embedding and the above estimate, one can derive that
\begin{align*}
	\|M_{\vec{j}}f\|_{L^{4\over3}(\Rd)}&\leq \left\|\left(\sum_{\vec{k}}|A_{\vec{j}}^{\vec{k}}f|^2\right)^{1\over2}\right\|_{L^{4\over3}(\Rd)}\\
	&\lesssim {2^{-\max\{j_1,\dots,j_d\}{{d-1}\over4}}}\left\|\left(\sum_{\vec{k}}|f*\Psi_{\vec{k}-\vec{j}}|^2\right)^{1\over2}\right\|_{L^{4\over3}(\Rd)}\\
	&\lesssim {2^{-\max\{j_1,\dots,j_d\}{{d-1}\over4}}}\|f\|_{L^{4\over3}(\Rd)},
\end{align*}
where we used multi-parameter Littlewood--Paley inequality in the last step. Using a similar argument inductively, we get
\begin{align*}
	\|M_{\vec{j}}f\|_{L^p(\Rd)}\lesssim {2^{-\max\{j_1,\dots,j_d\}\epsilon}}\|f\|_{L^p(\Rd)},
\end{align*}
for all $p>1$ and some $\epsilon>0$. For the case $p>2$, we can interpolate with the trivial $L^\infty$-bound. Now summing in $\vec{j}$ completes the proof
of Theorem \ref{thm1}.

\bigskip
\bigskip
\bigskip

{\bf Acknowledgments:}  J. Li is supported by ARC DP 220100285. C.-W. Liang and C.-Y. Shen are supported in part by NSTC through grant 111-2115-M-002-010-MY5. 
C.-W. Liang is also supported by MQ Cotutelle PhD scholarship. S. S. Choudhary is supported by National Center for Theoretical Sciences, National Taiwan University. Ji Li and Chong-Wei Liang would like to thank Professor Sanghyuk Lee for discussions on the spherical average. 
\bigskip


\bigskip

(S. S. Choudhary), National Center for Theoretical Sciences, No. 1 Sec. 4 Roosevelt Rd., National
Taiwan University, Taipei, 106, Taiwan.\\
{\it E-mail}: \texttt{surjeet19@ncts.ntu.edu.tw}\\

(J. Li) School of Mathematical and Physical Sciences, Macquarie University, NSW, 2109, Australia.\\ 
{\it E-mail}: \texttt{ji.li@mq.edu.au}\\

(C.-W. Liang) School of Mathematical and Physical Sciences, Macquarie University, NSW, 2109, Australia; and  Department of Mathematics, National Taiwan University, Taiwan.
\\ 
{\it E-mail}: \texttt{chongwei.liang@students.mq.edu.au;\ \ d10221001@ntu.edu.tw}\\
 
(C.-Y. Shen) Department of Mathematics, National Taiwan University, Taiwan.
\\ 
{\it E-mail}: \texttt{cyshen@ntu.edu.tw}\\

\vspace{0.3cm}


\begin{thebibliography}{99}

\bibitem{AE}
P. Auscher and M. Egert,
\textit{Boundary value problems and Hardy spaces for elliptic systems with block structure}, Progress in Mathematics, vol.~346, Birkh\"auser, Cham, 2023.
\vspace{-0.3cm}


\bibitem{MR0874045}
J. Bourgain, Averages in the plane over convex curves and maximal operators, J. Analyse Math. {\bf 47} (1986), 69--85.
  \vspace{-0.3cm}


\bibitem{MR0537803}
C.~P. Calder\'on, Lacunary spherical means, Illinois J. Math. {\bf 23} (1979), no.~3, 476--484.
 \vspace{-0.3cm}


\bibitem{CDK}
A. Chang, G. Dosidis and J. Kim, Nikodym sets and maximal functions associated with spheres, Rev. Mat. Iberoam., 41 (2025), no. 3, 1009--1056.
  \vspace{-0.3cm}
  
  
\bibitem{CF}
S.-Y. A. Chang and R. Fefferman, 
A continuous version of duality of $H^1$ with BMO on the bidisc,
Ann. of Math., (2) 112 (1980), no. 1, 179--201. 
  \vspace{-0.3cm}
 
 
 \bibitem{CMSGCIN}
M. Chen, S. Guo and T. Yang,
 A multi-parameter cinematic curvature,
 preprint, arXiv:2306.01606. 
  \vspace{-0.3cm}


\bibitem{Chr}
M. Christ, \textit{Weak type $(1,1)$ bounds for rough operators}, Ann. Math., 128 (1988), 19--42.
   \vspace{-0.3cm}
   


\bibitem{Cor}
A. C\'ordoba, The Kakeya maximal function and the spherical summation multipliers, Amer. J. Math., 99 (1977), no. 1, 1--22.
  \vspace{-0.3cm}

\bibitem{DR}   
   J. Duoandikoetxea and J. Rubio de Francia, \textit{Maximal and singular integral operators via
Fourier transform estimates}, Invent. Math., 84 (1986), no. 3, 541--561.
  \vspace{-0.3cm}


\bibitem{FS}  C. Fefferman and E.M. Stein, $H^p$ spaces of several variables, Acta Math., 129 (1972), no. 3-4, 137--193.
  \vspace{-0.3cm}

\bibitem{GT}
P. Ganguly and S.Thangavelu, On the lacunary spherical maximal function on the Heisenberg group, J. Funct. Anal. 280 (2021), no. 3, Paper No. 108832, 32 pp.

  \vspace{-0.3cm}


\bibitem{GHW}
A.  Govindan Sheri, J. Hickman and J. Wright, Lacunary maximal functions on homogeneous groups, J. Funct. Anal., 286 (2024), no. 3, Paper No. 110250, 25 pp.
  \vspace{-0.3cm}



\bibitem{Heo}
Y. Heo, \textit{Weak type estimates for some maximal operators on Hardy spaces}, Math.
Nachr., 280 (2007), no. 3, 281--289.
   \vspace{-0.3cm}


\bibitem{HJ}
J. Hickman and A. Jan\v{c}ar,
Spherical maximal estimates via geometry, 
Colloq. Math., 178 (2025), no. 1, 97--112. 
   \vspace{-0.3cm}


 \bibitem{HZ}
J. Hickman and J. Zahl,
Improved $L^p$ bounds for the strong spherical maximal operator, 	arXiv:2502.02795.
   \vspace{-0.3cm}


  \bibitem{MR4800578}
J. Lee, S. Lee and S. Oh, The elliptic maximal function, J. Funct. Anal. {\bf 288} (2025), no.~1, Paper No. 110693, 31 pp.      
\vspace{-0.3cm}


\bibitem{LLO}
J. Lee, S. Lee and S. Oh, $L^p$ bound on the strong spherical maximal function, Proc. Amer. Math. Soc., 153 (2025), no. 3, 1155--1167. 
   \vspace{-0.3cm}



\bibitem{MR1333890}
P. Mattila, {\it Geometry of sets and measures in Euclidean spaces}, Cambridge Studies in Advanced Mathematics, 44, Cambridge Univ. Press, Cambridge, 1995.        
\vspace{-0.3cm}


\bibitem{MR3617376}
P. Mattila, {\it Fourier analysis and Hausdorff dimension}, Cambridge Studies in Advanced Mathematics, 150, Cambridge Univ. Press, Cambridge, 2015. 
\vspace{-0.3cm}


\bibitem{MR1448717}
  A. Nevo and S. Thangavelu, Pointwise ergodic theorems for radial averages on the Heisenberg group, Adv. Math. {\bf 127} (1997), no.~2, 307--334.
   \vspace{-0.3cm}


\bibitem{Ober}
D.M. Oberlin, \textit{An endpoint estimate for some maximal operators}, Rev. Mat.
Iberoam., 12 (1996), no. 3, 641--652.
  \vspace{-0.3cm}


  \bibitem{MR4523247}
    J. Roos, A. Seeger and R. Srivastava, 
     \textit{Lebesgue space estimates for spherical maximal functions on
              {H}eisenberg groups},
   Int. Math. Res. Not. IMRN,
   24 (2022), 19222--19257. 
  \vspace{-0.3cm}

\bibitem{Sch1}
  W. Schlag, A generalization of Bourgain's circular maximal theorem, J. Amer. Math. Soc., 10 (1997), no. 1, 103--122. 
  \vspace{-0.3cm}

 
  \bibitem{Sch2}
W. Schlag, A geometric proof of the circular maximal theorem, Duke Math. J., 93 (1998), no. 3, 505--533.
   \vspace{-0.3cm}

   
\bibitem{MR2918096}
   A. Seeger and J. Wright,
    \textit{Problems on averages and lacunary maximal functions},
Banach Center Publ. 95 (2011), 235--250.
  \vspace{-0.3cm}

   

  
\bibitem{MR4693933}
R. Srivastava, 
{\it On the {K}or\'{a}nyi spherical maximal function on
              {H}eisenberg groups},
   Math. Ann., 
  388 (2024), 191--247.
    \vspace{-0.3cm}



  

\bibitem{MR0420116}
E.M. Stein, Maximal functions. I. Spherical means, Proc. Nat. Acad. Sci. U.S.A. {\bf 73} (1976), no.~7, 2174--2175.
\vspace{-0.3cm}
  
  
\bibitem{MR1232192}
    E.M. Stein,
    \textit{Harmonic analysis: real-variable methods, orthogonality, and
              oscillatory integrals},
    Princeton University Press, Princeton, NJ.,
    43 (1993), xiv+695.
\vspace{-0.3cm}
  
\bibitem{Wol}
T. Wolff, A Kakeya-type problem for circles, Amer. J. Math., 119 (1997), no. 5, 985--1026.
\vspace{-0.3cm}

\vspace{-0.4cm}


\bibitem{YLK} D. Yang, Y. Liang and L.D. Ky,
\textit{Real-variable theory of Musielak--Orlicz Hardy spaces},
Lecture Notes in Mathematics, vol.~2182, Springer, Cham, 2017.
\vspace{-0.3cm}

  
\end{thebibliography}
\end{document}